\documentclass[11pt]{amsart}
\usepackage{amssymb,amsmath}
\usepackage{eucal}
\usepackage{epstopdf}
\usepackage{mathrsfs}
\usepackage{enumerate}
\usepackage{fullpage} 
\usepackage{setspace}
\usepackage{xcolor}
\usepackage{dsfont}

\usepackage{todonotes}
\usepackage{hyperref}

\usepackage[OT2,T1]{fontenc}
\DeclareSymbolFont{cyrletters}{OT2}{wncyr}{m}{n}
\DeclareMathSymbol{\Sha}{\mathalpha}{cyrletters}{"58}

\newcommand{\X}{X}

\newcommand{\p}{p}

\newcommand{\oh}{\mbox{$\frac{1}{2}$}}

\DeclareFontFamily{OT1}{rsfs}{}
\DeclareFontShape{OT1}{rsfs}{n}{it}{<-> rsfs10}{}
\DeclareMathAlphabet{\mathscr}{OT1}{rsfs}{n}{it}

\newtheorem{prop}{Proposition}[section]
\newtheorem{thm}[prop]{Theorem}

\newtheorem{corollary}[prop]{Corollary}
\newtheorem{lemma}[prop]{Lemma}

\newtheorem*{defn*}{Definition}
\newtheorem{conj}{Conjecture}

\numberwithin{equation}{section}

\renewcommand{\Re}{{\mathfrak{Re}}}

 \newcommand{\mymod}[1]{(\operatorname{mod} #1)}
\DeclareMathOperator{\Tam}{Tam}
\DeclareMathOperator{\Gal}{Gal}

\DeclareMathOperator{\Cov}{Cov}
\DeclareMathOperator{\Var}{Var}

\DeclareFontFamily{U}{mathx}{}
\DeclareFontShape{U}{mathx}{m}{n}{<-> mathx10}{}
\DeclareSymbolFont{mathx}{U}{mathx}{m}{n}
\DeclareMathAccent{\widehat}{0}{mathx}{"70}
\DeclareMathAccent{\widecheck}{0}{mathx}{"71}

\begin{document}

\title[]{On distributions of $L'$-values and orders of Sha groups in families of quadratic twists}

\dedicatory{}

\author[P.-J. Wong]{Peng-Jie Wong}

\address{Department of Applied Mathematics\\
National Sun Yat-Sen University\\
Kaohsiung City, Taiwan}
\email{pjwong@math.nsysu.edu.tw}






\thanks{P.J.W. is currently supported by the NSTC grant 111-2115-M-110-005-MY3.}

\begin{abstract}
In this article, we aim to establish a prototype result regarding lower bounds of (joint) distributions of central $L'$-values through extending a method of {Radziwi\l\l} and Soundararajan of proving conditional bounds for distributions of central $L$-values (via the one-level density of low-lying zeros of involving $L$-functions). To illustrate this, we give several conditional bounds towards joint distributions of central $L'$-values and orders of Tate-Shafarevich groups in rank-one families of quadratic twists. As an application, we derive a simultaneous non-vanishing result for central $L'$-values in families of quadratic twists of triples of holomorphic modular forms.
\end{abstract}

\maketitle

\newcommand{\pj}[1]{{\color{blue} \sf  PJ: #1}}

\section{Introduction}

The non-vanishing and moments of $L$-functions at $s=\frac{1}{2}$ is an attractive research area in number theory. As an analogue of Selberg's central limit theorem on the normality of the distribution of $
\log |\zeta( \frac{1}{2} + it)|$, a beautiful conjecture of Keating and Snaith predicts analogous normality of the distribution of logarithms of central $L$-values of certain quadratic twists of elliptic curves (see \cite{KS}), and magnificent progress towards this conjecture has been made by {Radziwi\l\l} and Soundararajan \cite{RaSo15, RaSo}. Furthermore, the mean values of the derivatives of quadratic-twisted modular $L$-functions have been  settled by Bump-Friedberg-Hoffstein \cite{BFH} and M.R. Murty-V.K. Murty \cite{MuMu}, independently. (See also \cite{Iw,Mu11-CM,Mu11-TAMS,Sh}.) Moreover, the second moments of these $L'$-values have been established by Petrow \cite{Pe} (under GRH, based on the work of Soundararajan-Young \cite{SY}) and by Kumar-Mallesham-Sharma-Singh \cite{KMS} (unconditionally, based on the work of Li \cite{Li}). Inspired by these works, in this article, we shall extend the study of the above-mentioned conjecture of Keating and Snaith to the logarithms of central $L'$-values of quadratic-twisted modular forms as follows.

Let $f$ be a holomorphic cuspidal newform of even weight $k$ on the congruence subgroup $\Gamma_0(N)$ with the trivial central character. The associated (modular) $L$-function is defined by
$
L(s,f) =\sum_{n=1}^\infty \frac{\lambda_{f}(n)}{n^s},
$ 
for $\Re(s)>1$, which extends to an entire function. (Recall that the celebrated work of Deligne implies $|\lambda_{f}(n)| \le d_2(n)$ for all $n$.) Moreover, setting
$
\Lambda(s,f) = (\frac{\sqrt{N}}{2\pi})^s \Gamma(s+\frac{k-1}{2}) L(s,f),
$
one has the functional equation
$
\Lambda(s,f) = \epsilon_f\Lambda(1-s,f),
$ 
where $\epsilon_f =\pm 1$ is the root number of $f$. Furthermore, throughout our discussion, we will let $d$ denote a fundamental discriminant coprime to $2N$, and $\chi_d =(\frac{d}{\cdot})$ be the associated primitive quadratic (Dirichelt) character. As each quadratic twist $f\otimes \chi_d$ defines a newform on $\Gamma_0(N|d|^2)$, there is an associated twisted $L$-function defined by
$$
L(s,f\otimes \chi_d) =\sum_{n=1}^\infty \frac{\lambda_{f}(n)\chi_d(n)}{n^s}
$$
whose completed $L$-function
$$
\Lambda(s,f\otimes \chi_d) =  \bigg(\frac{\sqrt{N}|d|}{2\pi}\bigg)^s  \Gamma(s+\frac{k-1}{2}) L(s,f\otimes \chi_d)
$$
extends to an entire function and satisfies $$
\Lambda(s,f\otimes \chi_d) = \epsilon_f(d)\Lambda(1-s,f\otimes \chi_d)\quad\text{with}\quad 
\epsilon_f(d) = \epsilon_f\chi_d(-N).
$$
Clearly, $L(\frac{1}{2},f\otimes \chi_d) = 0$ if $\epsilon_f(d) = -1$. Thus, in what follows, we shall consider 
$$
\mathcal{F}_f = \{d :\text{$d$ is a fundamental discriminant with $(d, 2N)$ = 1 and $\epsilon_f(d) = -1$} \}
$$
and study $L'(\frac{1}{2},f\otimes \chi_d)$ over $\mathcal{F}_f$.

Our main theorem is the following prototype result that particularly illustrates how to extend the general principle of {Radziwi\l\l} and Soundararajan \cite{RaSo} to study (joint) distributions of central $L'$-values via the one-level density of low-lying zeros of involving $L$-functions.

\begin{thm}\label{main-thm-M-forms}
 Fix $M\in\{1,2,3\}$. Let $f_j$ be distinct holomorphic cuspidal newforms of even weight $k_j$ on the congruence subgroup $\Gamma_0(N_j)$ with trivial central character, and let $\mathcal{F} =\cap_{j=1}^{M}\mathcal{F}_{f_j} $  Assume the generalised Riemann hypothesis (GRH) for all twisted $L$-functions $L(s,f_j \otimes \chi)$ with  Dirichlet characters $\chi$. Then for any fixed $\underline{\alpha}= (\alpha_1,\ldots,\alpha_M)$ and $\underline{\beta}=(\beta_1,\ldots,\beta_M)$, as $ X \rightarrow \infty$, 
\begin{align*}
 \begin{split}
&\#\bigg\{ d\in\mathcal{F}, X< |d|\le 2X :   \frac{\log |L'(\oh,f_j \otimes \chi_d)| -\frac{1}{2} \log \log |d|}{\sqrt{ \log \log |d|}} \in (\alpha_j,\beta_j)\enspace \forall 1\le j\le M \bigg\}\\
&\ge C_M (\Psi_M(\underline{\alpha},\underline{\beta})  +o(1) )
\# \{ d \in\mathcal{F} :  X< |d|\le 2X\},
 \end{split}
\end{align*}
where 
 \begin{equation}\label{def-Psi_M}
C_M
= 1-\frac{M}{4} 
\quad
\text{and}
\quad 
\Psi_M(\underline{\alpha},\underline{\beta}) 
=\int_{(\alpha_1,\beta_1)\times\cdots\times(\alpha_M,\beta_M)} \frac{1}{2\pi}e^{-\frac{1}{2}{\bf v}^{\mathrm{T}}{\bf v}} d{\bf v}.
\end{equation}
\end{thm}

\noindent {\bf Remark.} Our theorem is inspired by the work of Petrow \cite[Theorem 2.2]{Pe}, who proved a mean value estimate for $L'(\oh,f_1 \otimes \chi_d)L'(\oh,f_2 \otimes \chi_d)$ under GRH. As remarked by Petrow, the results involving two newforms are particularly interesting since the asymptotics for the analogous moment as well as the lower bound towards Keating-Snaith's conjecture mentioned previously \emph{without} derivatives seems completely out of reach by the techniques developed to date. Also, as a  direct consequence, at least a quarter of central $L'$-values of triples of holomorphic modular forms are \emph{simultaneously} non-vanishing. Such a conditional non-vanishing result is somewhat surprising and seemingly hard to obtain via calculating the moments (as there are three degree-two $L$-functions involved).

Let $E^j$ be (pairwise) non-isogenous elliptic curves over $\Bbb{Q}$. By the modularity theorem, there are holomorphic cuspidal newforms $f_j=f_{E^{j}}$ such that $L(s,E_d^j) = L(s,f_j \otimes \chi_d)$, where the left is the $L$-function associated to the quadratic twist $E_d^j$ of $E^j$ by $\chi_d$. Recall that the famous theorem of Faltings implies that two elliptic curves over $\Bbb{Q}$ are isogenous if and only if they have the same $L$-factors at almost all primes  (see \cite{Fa}). Thus, by the strong multiplicity one theorem, we know that $f_j$ must be distinct when $E^j$ are (pairwise) non-isogenous elliptic curves. Consequently, we have the following corollary.


\begin{corollary}\label{EC-coro} Fix $M\in\{1,2,3\}$. In the same notation as above, for non-isogenous elliptic curves  $E^j/\Bbb{Q}$, set $\mathcal{F} =\cap_{j=1}^{M}\mathcal{F}_{f_j} $. Under GRH, for any fixed $\underline{\alpha}= (\alpha_1,\ldots,\alpha_M)$ and $\underline{\beta}=(\beta_1,\ldots,\beta_M)$,
\begin{align}\label{E-main-thm-eq}
 \begin{split}
&\#\bigg\{ d\in\mathcal{F}, X< |d|\le 2X :   \frac{\log |L'(\oh,E^j_d)| -\frac{1}{2} \log \log |d|}{\sqrt{ \log \log |d|}} \in (\alpha_j,\beta_j)\enspace \forall 1\le j\le M \bigg\}\\
&\ge C_M (\Psi_M(\underline{\alpha},\underline{\beta})  +o(1) )
\# \{ d \in\mathcal{F} :  X< |d|\le 2X\},
 \end{split}
\end{align}
as $ X \rightarrow \infty$, where $C_M$ and $\Psi_M$ are defined as in \eqref{def-Psi_M}.
\end{corollary}

It is worthwhile mentioning that when $M=1$ and $E=E^1$ (so $E_d = E^1_d$), the factor $C_1 =\frac{3}{4}$ appearing in \eqref{E-main-thm-eq} agrees with the positive proportion of non-vanishing $L'(\frac{1}{2},E_d)$ obtained by Heath-Brown \cite{H-B} under GRH. Also, this special instance (of odd orthogonal symmetry) can be seen as a complementary result for the recent work of {Radziwi\l\l} and Soundararajan \cite{RaSo} that proves a conditional lower bound (upon a factor of $\frac{1}{4}$) towards the conjecture of Keating and Snaith on the normality of $L(\frac{1}{2},E_d)$ over the even orthogonal symmetry family
\begin{equation}\label{def-E}
\mathcal{E} = \{d :\text{$d$ is a fundamental discriminant with $(d, 2N)$ = 1 and $\epsilon_E(d) = +1$} \},
\end{equation}
where $N=N_E$ is the conductor of $E$, and $\epsilon_E(d)$ is the root number of $L(s,E_d)$. 

We shall remark that all of these are rooted in Goldfeld's conjecture \cite{Go} predicting that $50\%$ of $L(s,E_d)$ have analytic rank zero and  $50\%$ of $L(s,E_d)$ have analytic rank one, which follows from the Birch and Swinnerton-Dyer conjecture (BSD) and the recent work of Smith \cite{Si} if $E$ has full rational $2$-torsion and no rational cyclic subgroup of order four. Unfortunately, as the approach of Smith is quite of an algebraic nature, it does not seem easy to combine it with the analytic principle of {Radziwi\l\l}-Soundararajan to obtain some further unconditional results.

In \cite[Conjecture 1]{RaSo}, {Radziwi\l\l} and Soundararajan further formulated a conjecture for the distribution of orders of Tate-Shafarevich groups $\Sha(E_d)$ of $E_d$ that as $d$ ranges over $\mathcal{E}$, the distribution of $\log(|\Sha(E_d)|/ \sqrt{|d|})$ is approximately Gaussian, with mean $\mu(E)\log \log |d|$ and variance $\sigma(E)^2 \log \log |d|$.  
  As this manifests as the ``rank-zero'' case, one may ask whether an analogue for  the ``rank-one'' family (more precisely, replacing $\mathcal{E}$ by $\mathcal{F}$). With this in mind, we begin by recalling that the rank-one BSD asserts that if $L(\frac{1}{2},E_d)= 0$ and $L'(\frac{1}{2},E_d)\neq 0$, then 
$$
L'(\oh,E_d)= \frac{ |\Sha(E_d)| R(E_d) \Omega(E_d)\Tam(E_d)}{|E_d(\Bbb{Q})_{\mathrm{tors}}|^2},
$$
where $R(E_d)$ is the regulator of $E_d$, $|E_d(\Bbb{Q})_{\mathrm{tors}}|$ denotes the order of the rational torsion group of $E_d$, $\Omega(E_d)$ is the real period of a minimal model for $E_d$, and 
$
\Tam(E_d) =\prod_p T_p(d)
$  
is the product of the Tamagawa numbers. Additionally, when $L(\frac{1}{2},E_d)= 0$ and $L'(\frac{1}{2},E_d)\neq 0$, one can consider 
\begin{equation}\label{def-SEd}
S(E_d) = L'(\oh, E_d)  \frac{|E_d(\Bbb{Q})_{\mathrm{tors}}|^2}{R(E_d)\Omega(E_d)\Tam(E_d)}, 
\end{equation}
which is a more appropriate object for analytic study.

In the light of Keating-Snaith's conjecture, \cite[Conjecture 1]{RaSo}, and Corollary \ref{EC-coro}, we consider the following conjecture regarding  the joint distribution of central $L'$-values and orders of Tate-Shafarevich groups of $E_d$.

\begin{conj}\label{R1-conj}
Let $E/\Bbb{Q}$ be an elliptic curve given in Weierstrass form $y^2 = F(x)$ for a monic cubic integral polynomial $F$, and let $K$ denote the splitting field of $F$ over $\Bbb{Q}$. Let $c(g)$ be the integer so that $c(g)-1$ is the number of fixed points of $g\in \Gal(K/\Bbb{Q})$, and set
$$
\mu(E) =-\frac{1}{2} -\frac{1}{|G|}\sum_{g\in G} \log c(g) 
\quad\text{and}\quad
\sigma(E) =1 +\frac{1}{|G|}\sum_{g\in G} (\log c(g))^2. 
$$
Then the following hold.\\
\noindent (i) 
As $d$ ranges over $\mathcal{F}$, the joint distribution of $\log |L'(\oh,E_d)|$ and $\log(|\Sha(E_d)|R(E_d)/ \sqrt{|d|})$
 is approximately bivariate. More precisely, for any fixed $\underline{\alpha}= (\alpha_1,\alpha_2)$ and $\underline{\beta}=(\beta_1,\beta_2)$, as $X\rightarrow \infty$,
\begin{align*}
\begin{split}
\#\bigg\{ d\in\mathcal{F},
20 <|d| \le X :\, & \frac{ \log |L'(\oh,E_{d})| - \frac{1}{2}\log \log |d| }{\sqrt{ \log\log |d|}}\in (\alpha_1,\beta_1),\\
& \frac{\log (|\Sha(E_d)| R(E_d)/\sqrt{|d|} ) - (\mu(E)+1) \log\log |d|}{\sqrt{\sigma(E)^2 \log\log |d|}}
\in (\alpha_2,\beta_2)\bigg\} 
\end{split}
\end{align*}
is asymptotic to
$
(\Xi_E(\underline{\alpha},\underline{\beta}) +o(1)) \# \{ d \in\mathcal{F} : 20< |d| \le X \},  
$
where
\begin{equation}\label{def-XiE-KE}
\Xi_E(\underline{\alpha},\underline{\beta}) 
=\int_{(\alpha_1,\beta_1)\times(\alpha_2,\beta_2)} \frac{1}{2\pi \sqrt{\det(\mathfrak{K}_E)}}e^{-\frac{1}{2}{\bf v}^{\mathrm{T}}\mathfrak{K}_E^{-1}{\bf v}} d{\bf v}
\quad\text{with}\quad 
\mathfrak{K}_E 
=\begin{pmatrix} 
1 & \sigma(E)^{-1}  \\
\sigma(E)^{-1} & 1  \\
\end{pmatrix}. 
\end{equation}
\noindent (ii) Consequently, as $d$ ranges over $\mathcal{F}$, the distribution of $\log(|\Sha(E_d)| R(E_d)/ \sqrt{|d|})$
 is approximately Gaussian, with mean $(\mu(E)+1)\log \log |d|$ and variance $\sigma(E)^2 \log \log |d|$. 
\end{conj}
Denoting $n_K$ the degree of $K$, one has the following table of explicit values of $\mu(E)$ and  $  \sigma(E)^2$:
\begin{table*}[h!] 
\begin{tabular}{|c||c|c|c|c|} 
 \hline
 $n_K$ & $1$ & $2$ & $3$ & $6$    \\ 
  \hline
  $\mu(E)$ & $ -\frac{1}{2} -2 \log 2$ & $-\frac{1}{2}  -\frac{3}{2} \log 2$ & $-\frac{1}{2}  -\frac{2}{3} \log 2$ & $-\frac{1}{2}  -\frac{5}{6} \log 2$    \\ 
   \hline
  $  \sigma(E)^2 $ & $1 + 4(\log 2)^2$ & $1 + \frac{5}{2} (\log 2)^2$ & $1 + \frac{4}{3} (\log 2)^2$ & $1 + \frac{7}{6} (\log 2)^2$    \\ 
 \hline
\end{tabular}
\end{table*}

Our second main result is the following conditional lower bound towards Conjecture \ref{R1-conj}.

\begin{thm}\label{main-thm}
Assume GRH for the family of twisted $L$-functions $L(s,E \otimes \chi)$ with all Dirichlet characters $\chi$. Let $f_E$ be the newform corresponding to $E$, and let $\mathcal{F} = \mathcal{F}_{f_E}$. Then, for any fixed $ \underline{\alpha}=(\alpha_1,\alpha_2)$ and  $\underline{\beta}=(\beta_1,\beta_2)$,  as $ X \rightarrow \infty$, 
\begin{align*}
\begin{split}
\#\bigg\{ d\in\mathcal{F},
X< |d|\le 2X :\, & \frac{ \log |L'(\oh,E_{d})| - \frac{1}{2}\log \log |d| }{\sqrt{ \log\log |d|}}\in (\alpha_1,\beta_1),\\
& \frac{\log (|S(E_d)|R(E_d)/\sqrt{|d|} ) - (\mu(E)+1) \log\log |d|}{\sqrt{\sigma(E)^2 \log\log |d|}}
\in (\alpha_2,\beta_2)\bigg\} 
\end{split}
\end{align*}
is greater than or equal to
$$
\frac{3}{4} (\Xi_E(\underline{\alpha},\underline{\beta})+o(1) )
\# \{ d \in\mathcal{F} :  X< |d|\le 2X \}, 
$$
where $\Xi_E(\underline{\alpha},\underline{\beta})$ is defined as in \eqref{def-XiE-KE}. Furthermore, suppose that the Birch and Swinnerton-Dyer conjecture holds for elliptic curves with analytic rank one. Then the above assertion remains true with $S(E_d)$ being replaced by $|\Sha(E_d)|$.
\end{thm}

\noindent {\bf Remark.} 
 (i)  In comparison to \cite[Conjecture 1]{RaSo}, Conjecture \ref{R1-conj} predicts that  as $d$ ranges over $\mathcal{F}$, the distribution of $\log(|\Sha(E_d)| R(E_d)/ \sqrt{|d|})$
 is approximately Gaussian, with mean $ (\mu(E)+1)\log \log |d|$ and variance $\sigma(E)^2 \log \log |d|$. It seems hard to remove $R(E_d)$ from the conjecture as it has been shown in \cite[Proposition 2.2]{DR} that if $j(E)\neq 0,1728$, then one has $R(E_d) \gg_E \log |d|$. 

\noindent (ii) Our proof of Theorems \ref{main-thm-M-forms} and \ref{main-thm} are based on the works of {Radziwi\l\l} and Soundararajan \cite{RaSo15,RaSo}, which led us to refine \cite[Proposition 3]{RaSo} as in Proposition \ref{key-prop-2}, which particularly gives us the flexibility to introduce a sieve parameter $v$ for $d$ (cf. \cite[Proposition 1]{RaSo15}). Also, we require Proposition \ref{new-prop1} to express central $L'$-values in terms of certain Dirichlet polynomials and zero sums of the $L$-functions in consideration.

\noindent (iii)  A key new input of this article is to invoke the Cram\'er-Wold device (see, e.g., \cite[Theorem 29.4]{Billingsley}), which allows us to derive the desired multivariate normal distributions by studying weighted moments of real combinations of associated Dirichlet polynomials. This idea was also used in the joint work \cite{HW} with Hsu, where the Cram\'er-Wold device was implemented inexplicitly to establish a ``log-independence'' among Dirichlet $L$-functions (over the critical line). We shall therefore give a detailed argument in the proofs of Lemmata \ref{lemma-G} and \ref{lemma-H}.

\noindent (iv) By adapting the method developed in \cite{RaSo15}, it may be possible to obtain unconditional upper bounds for Theorem \ref{main-thm-M-forms} (when $M=1$) and  Theorem \ref{main-thm} since the $(1+\varepsilon)$-th moments of the involving $L$-functions could be calculated.

\noindent (v)  Recently, under GRH, Bui, Evans, Lester, and Pratt \cite{BELP} proved analogues of Keating-Snaith's conjecture (for mollified central values of both $L(s,f_1\otimes\chi)$ and $L(s,f_2\otimes\chi)$ over Dirichlet characters $\chi$ modulo $q$) with a full asymptotic. As discussed in \cite[Sec. 1.4]{BELP}, they directly computed an asymptotic formula for the involving joint distribution, which requires a more complicated calculation in contrast to our argument.

The rest of this article is arranged as follows. Section \ref{notation} will collect preliminaries and results for proving Theorems \ref{main-thm-M-forms} and \ref{main-thm}. The proof of Theorem \ref{main-thm-M-forms} will be given in Section \ref{proof-of-main-thm}. Three key propositions, Propositions \ref{key-prop-2}, \ref{M-forms-moments}, and \ref{P-C-moments}  will be proved in Sections \ref{proof-key}, \ref{proof-2.2}, and \ref{pf-PC}, respectively. In Section \ref{rmkRS}, we will discuss an extension of \cite[Conjecture 1]{RaSo} in the same spirit of Conjecture \ref{R1-conj}.

\section{Notation and the key propositions}\label{notation}

Fix $M\in\Bbb{N}$. To extend the argument of \cite{RaSo}, we let $N_0$ denote the lcm of 8 and $N_{f_1},\ldots, N_{f_M}$. Set $\kappa=\pm 1$, and let $a\enspace\mymod{N_0}$ denote a residue class with $a \equiv 1$ or $5\enspace\mymod{8}$. In addition, we shall assume that $\kappa$ and $a$ are such that for any fundamental discriminant $d$ with sign $\kappa$, 
satisfying $d \equiv a\enspace \mymod{N_0}$, the root number $\epsilon_{f_j}(d)$ is $-1$ for every $1\le j\le M$. For these $\kappa$ and $a$, we set
$$
\mathcal{F} (\kappa, a) = \{d \in \mathcal{F} : \kappa d > 0,\enspace d \equiv a\enspace \mymod{N_0}\}
$$
so that $\mathcal{F} =\cap_{j=1}^{M}\mathcal{F}_{f_j} $ is the union of $\mathcal{F} (\kappa, a)$. Note that the imposed congruence condition on $d$ forces  $d \equiv 1\enspace \mymod{4}$, and thus $d\in \mathcal{F} (\kappa, a)$ must be square-free as $d$ is a fundamental discriminant. Moreover, we note that for $d\in \mathcal{F} (\kappa, a) $, the values 
$\chi_d(-1)$, $\chi_d(2)$, $\chi_d(p)$, with $p\mid N_0$, are fixed.

Let $h$ be an even smooth function, satisfying $|h(t)| \ll 1/(1+t^2)$ for all real $t$, and suppose that its  Fourier transform 
$
\widehat{h}(\xi)= \int_{\Bbb{R}} h(t) e^{-2\pi i\xi t}dt
$
is compactly supported. (Particularly, we can take the Fej\'er kernel
$
h(t) = ( \frac{\sin(\pi t)}{\pi t} )^2
$
with $\widehat{h}(\xi) =\max(1 - |\xi|, 0)$.) Also, we let $\Phi$ be a smooth non-negative function compactly supported in $[\frac{1}{2},\frac{5}{2}]$ such that $\Phi(t)=1$ on $[1,2]$.
Processing the argument of \cite[Proposition 1 and p. 1055]{RaSo15} verbatim then yields that for $n,v\in\Bbb{N}$ coprime to $N_0$, with $(n,v)=1$ and $v$ being square-free, such that $v\sqrt{n}\le X^{\frac{1}{2}-\varepsilon}$, one has
\begin{equation}\label{key-prop-RS-original}
\sum_{\substack{ d\in\mathcal{F}(\kappa,a)\\ v\mid d}}  \chi_d(n) \Phi\left(\frac{\kappa d}{X}\right)
= \delta(n=\square)
\frac{X}{vN_0} \prod_{p\mid n v} \left(1 + \frac{1}{p} \right)^{-1}
 \prod_{p\nmid N_0 } \left(1 - \frac{1}{p^2} \right) \widehat{\Phi}(0)
 + O(X^{\frac{1}{2}+\varepsilon}n^{\frac{1}{2}}),
\end{equation}
where $\delta(n=\square)=1$ if $n$ is a square, and $\delta(n=\square)=0$ otherwise. 
Furthermore, we have the following generalisation of \cite[Proposition 2]{RaSo} for modular $L$-functions.

\begin{prop}\label{key-prop-2}
Let $f_j$ be newforms. Let $h$ be an even smooth function such that $h(t) \ll (1 + t^2)^{-1}$, and its Fourier transform is compactly supported in $[-1, 1]$. Let $L \ge 1$ be real, and let $\ell$ and $v$ be positive integers coprime to $N_0$ such that $v$ is square-free and $(\ell, v) = 1$.  Assume further that $e^{\frac{L}{4}}\ell^{\frac{1}{2}} \le X^{\frac{1}{2}-3\varepsilon}$
and $v\le X^{\varepsilon}$. Under GRH for all $ L(s, f_j \otimes \chi_d \otimes \chi)$ with $d\in \mathcal{F}(\kappa,a)$ and Dirichlet characters $\chi$ modulo $N_0$, we let  $\oh+ i\gamma_{j,d}$ range over the non-trivial zeros of $ L(s,f_j \otimes \chi_d)$, and set
$$
\mathcal{S}_{\kappa,a, v}=\sum_{j=1}^M \mathcal{S}_{j,\kappa,a,v}, \quad\text{with}\quad
\mathcal{S}_{j,\kappa,a} 
=  \sum_{\substack{ d\in\mathcal{F}(\kappa,a)\\ v\mid d}} \sum_{\gamma_{j,d}} h\left( \frac{\gamma_{j,d} L}{2\pi}\right)\chi_d(\ell) \Phi\left(\frac{\kappa d}{X}\right).
$$
If $\ell$ is not a square nor a prime times a square, then
$$
\mathcal{S}_{\kappa,a , v} \ll MX^{\frac{1}{2}+3\varepsilon}\ell^{\frac{1}{2}} e^{\frac{L}{4}}.
$$
If $\ell$ is a square, then
\begin{align*}
\mathcal{S}_{\kappa,a , v} &= 
\frac{MX}{v N_0} \prod_{p\mid \ell v  } \left(1 + \frac{1}{p} \right)^{-1}
 \prod_{p\nmid N_0 } \left(1 - \frac{1}{p^2} \right) \widehat{\Phi}(0)
 \left( \frac{2\log X}{L} \widehat{h}(0)  +\frac{h(0)}{2}+ O(L^{-1}) \right)  \\
&+O( MX^{\frac{1}{2}+3\varepsilon}\ell^{\frac{1}{2}} e^{\frac{L}{4}}).
\end{align*}
Lastly, if $\ell$ is a prime $q$ times a square, then
$$
\mathcal{S}_{\kappa,a, v} \ll 
\frac{MX}{ v LN_0} \frac{\log q}{\sqrt{q}}  \prod_{p\mid \ell v } \left(1 + \frac{1}{p} \right)^{-1}
+MX^{\frac{1}{2}+3\varepsilon}\ell^{\frac{1}{2}} e^{\frac{L}{4}}.
$$
\end{prop}

We define
\begin{equation}\label{def-w}
\mathcal{P}_f(d; x) = \sum_{p\le x} \frac{\lambda_{f}(p) \chi_d(p)}{p^{\frac{1}{2}}} w(p),
\quad\text{with}\quad
w(p) = \frac{1}{p^{\frac{1}{\log x}}}   \frac{\log (x/p)}{\log x} .
\end{equation}
For $\mathbf{a} = (a_1,\ldots,a_M) \in\Bbb{R}^M$, we set
$
\mathcal{P}_{\mathbf{a}}(x;d) 
= \sum_{j=1}^M a_j\mathcal{P}_{f_j}(d; x).
$

The proof of Theorem \ref{main-thm} relies on the ``method of moments'' asserting roughly that normal random variables are uniquely determined by their moments  (see, e.g. \cite{FS}), which led us to the following proposition.

\begin{prop}\label{M-forms-moments}
Let $k\in \Bbb{Z}^+$, and set $x= X^{1/ \log \log \log X}$. Then for $X$ sufficiently large, we have
\begin{equation}\label{M-P-moments}
\sum_{d\in\mathcal{F}(\kappa,a)}   \mathcal{P}_{\mathbf{a}}(x;d)^k\Phi\left(\frac{\kappa d}{X}\right)
=  \sum_{d\in\mathcal{F}(\kappa,a)} \Phi\left(\frac{\kappa d}{X}\right) \bigg( \sum_{j=1}^M a_j^2\log\log X\bigg)^{\frac{k}{2}}(M_k +o(1)),
\end{equation}
where $M_k$ is the $k$-th standard Gaussian moment, i.e. $M_k=  \frac{k!}{2^{k/2} (k/2)!}$ if $k$ is even, and $M_k=0$ otherwise.
Moreover, for any $L\ge 1$ such that $e^L\le X^2$, we have 
\begin{align}\label{M-P-h-moments}
 \begin{split}
& \sum_{d\in\mathcal{F}}  \mathcal{P}_{\mathbf{a}}(x;d)^k\sum_{j=1}^M \sum_{\gamma_{j,d}} h\left( \frac{\gamma_{j,d} L}{2\pi}\right)\Phi\left(\frac{\kappa d}{X}\right)\\
& =
\frac{ M X}{N_0}
 \prod_{p\nmid N_0 } \left(1 - \frac{1}{p^2} \right) \widehat{\Phi}(0)
 \left( \frac{2\log X}{L} \widehat{h}(0)  +\frac{h(0)}{2}+ O(L^{-1}) \right) \bigg( \sum_{j=1}^M a_j^2\log\log X\bigg)^{\frac{k}{2}}(M_k +o(1))\\
&+O_M(X^{\frac{1}{2}+\varepsilon} e^{\frac{L}{4}}),
  \end{split}
\end{align}
where the implied constant depends on $M$.
\end{prop}

In the remaining part of this section, we let $E/\Bbb{Q}$ be an elliptic curve given by the model $y^2 = F(z)$, where $F$ is a monic cubic integral polynomial. Let $c(p)$ be the integer so that $c(p)-1 $ is the number of solutions of $F(z)\equiv 0\enspace\mymod{p}$. Recall that for $p$ not dividing the discriminant $\mathrm{disc}(F)$ of $F$, $T_p(d)= c(p)$ if $p\mid d$, and $T_p(d)=1$ otherwise.
Following \cite{RaSo15}, for $x> \log X > \max\{ N_0, |\mathrm{disc}(F)| \}$, we set
$$
\mathcal{C}(d;x)
=\sum_{\log X\le p \le x} C_{p}(d), \quad\text{with}\quad
C_{p}(d)
=
\log T_p(d)  -\frac{1}{p+1}\log c(p)
=
\begin{cases}
 \frac{p}{p+1}\log c(p) &\mbox{if $p\mid d$;}\\
 -\frac{1}{p+1}\log c(p) &\mbox{if $p\nmid d$.}
\end{cases}
$$
We also recall the following Mertens type estimates, involving $c(p)$, from \cite[Lemma 4]{RaSo15} (which is a consequence of the Chebotarev density theorem):
\begin{equation}\label{est-cp}
\sum_{p\le y} \frac{\log c(p) }{p} =\left(-\mu(E) -\oh\right)\log\log y +O(1)
\end{equation}
and
\begin{equation}\label{est-cp2}
\sum_{p\le y} \frac{(\log c(p))^2 }{p} =\left(\sigma(E)^2 -1\right)\log\log y +O(1).
\end{equation}
In addition, denoting $f=f_E$ the holomorphic cuspidal newform corresponding to $E$, we take $\mathcal{F}=\mathcal{F}_f$ and consider $\mathcal{P}(d; x) = \mathcal{P}_f(d; x)$. We have the following proposition regarding moments of real linear combinations of $\mathcal{P}(d; x)$ and $\mathcal{P}(d; x)-\mathcal{C}(d;x)$.

\begin{prop}\label{P-C-moments}
Let $k\in \Bbb{Z}^+$ be  fixed, and put $x= X^{1/ \log \log \log X}$.  Fix $b,c \in\Bbb{R}$  Then for $X$ sufficiently large, we have
\begin{align}\label{P-C-moments-1} 
 \begin{split}
&\sum_{d\in\mathcal{F}(\kappa,a)}  (b \mathcal{P}(d; x) + c( \mathcal{P}(d; x)- \mathcal{C}(d;x)))^k\Phi\left(\frac{\kappa d}{X}\right) \\
&=  \sum_{d\in\mathcal{F}(\kappa,a)} \Phi\left(\frac{\kappa d}{X}\right) \left( ( b^2 + 2bc + c^2\sigma(E)^2)\log\log X\right)^{\frac{k}{2}}(M_k +o(1)).
 \end{split}
\end{align}
Moreover, for any $L\ge 1$ such that $e^L\le X^{2-10\varepsilon}$, we have 
\begin{align}\label{P-C-h-moments}
 \begin{split}
& \sum_{d\in\mathcal{F}(\kappa,a)}  (b \mathcal{P}(d; x) + c( \mathcal{P}(d; x)- \mathcal{C}(d;x)))^k \sum_{\gamma_{d}} h\left( \frac{\gamma_{d} L}{2\pi}\right)\Phi\left(\frac{\kappa d}{X}\right)\\
& =
  \frac{X}{N_0}
 \prod_{p\nmid N_0 } \left(1 - \frac{1}{p^2} \right) \widehat{\Phi}(0)
 \left( \frac{2\log X}{L} \widehat{h}(0)  +\frac{h(0)}{2}+ O(L^{-1}) \right) \\
&\times\left(( b^2 + 2bc + c^2\sigma(E)^2)\log\log X\right)^{\frac{k}{2}}(M_k +o(1))+O(X^{\frac{1}{2}+\varepsilon} e^{\frac{L}{4}}),
  \end{split}
\end{align}
where the implied constants depend on $b,c$.
\end{prop}

\section{Proof of the main theorem}\label{proof-of-main-thm}

To prove Theorem \ref{main-thm-M-forms}, we start with the following proposition that gives an approximation for the central $L'$-values, which is similar to \cite[Proposition 1]{RaSo}. Our proof requires a modification of the one due to {Radziwi\l\l} and Soundararajan, and the key trick is to convert the derivative of an $L$-function to the one without derivative. As pointed out by Quanli Shen, such a trick of converting was also used by Kirila in \cite[Eq. (3.3)]{Ki} for studying discrete
moments of the derivative of the Riemann zeta function (over its non-trivial zeros).

\begin{prop}\label{new-prop1}
Let $d\in\mathcal{F}$, and let $ 3\le x \le |d|$. Assume
GRH for $L(s,f\otimes \chi_d)$, and let $\oh+ i\gamma_{d}$ run over 
the non-trivial zeros of $ L(s,f\otimes \chi_d)$.  If  $L'(\frac{1}{2},f\otimes \chi_d)$ is non-vanishing, then
\begin{align}\label{L'-exp}
\log | L'(\oh ,f\otimes \chi_d)|&=  \mathcal{P}_f(d; x) +\oh \log\log x +O(1)
+O\bigg( \frac{ \log |d|}{\log x}+\sum_{\gamma_d \neq 0} 
\log \bigg(1 + \frac{ 1 }{ (\gamma_d \log x)^2} \bigg)  \bigg),
\end{align}
where $\mathcal{P}_f(d; x)$ is defined as in \eqref{def-w}.
\end{prop}


\begin{proof}
First, from the logarithmic derivative of the Hadamard factorisation of $L(s,f\otimes\chi_d)$, it follows that
$$
-\Re \frac{L'}{L}(\sigma +it,f\otimes\chi_d)
 = \log |d| -
\sum_{\gamma_d}  \frac{\sigma- \frac{1}{2} }{(\sigma -\frac{1}{2})^2+ (t-\gamma_d)^2} +O_{k,N}(1),
$$
for $\frac{1}{2}\le\sigma\le 1$ and $|t|\le 1$. Integrating $\sigma$ from $\frac{1}{2}$ to $\sigma_0 =\frac{1}{2}+\frac{1}{\log x}$ then yields
\begin{align*}
&\log |L(\oh +it,f\otimes \chi_d)| -\log |L(\sigma_0 +it,f\otimes \chi_d)|\\
&= (\sigma_0 -\oh) (\log |d| +O(1) )
-\frac{1}{2}\sum_{\gamma_d} \log \frac{ (\sigma_0 -\frac{1}{2})^2 +(t-\gamma_d)^2  }{ (t-\gamma_d)^2}.
\end{align*}
By our assumption, $L(s,f\otimes \chi_d )$ has a simple zero at $s=\frac{1}{2}$ for $d\in\mathcal{F}$. Hence, by isolating such a zero, we can write
\begin{align*}
&\log \bigg|\frac{ L(\oh +it,f\otimes \chi_d)}{t}\bigg| -\log |L(\sigma_0 +it,f\otimes \chi_d)|\\
&= \frac{1}{\log x} (\log |d| +O(1) )
-\frac{1}{2}\log\bigg(\bigg(\frac{1}{\log x}\bigg)^2 +t^2  \bigg)
-\frac{1}{2}\sum_{\gamma_d \neq 0} \log \frac{ (\frac{1}{\log x})^2 +(t-\gamma_d)^2  }{ (t-\gamma_d)^2}.
\end{align*}
As $L(\frac{1}{2},f\otimes \chi_d )=0$, making $t\rightarrow 0$, we then arrive at
\begin{align}\label{log-L-prime}
 \begin{split}
&\log | L'(\oh ,f\otimes \chi_d)| -\log |L(\sigma_0 ,f\otimes \chi_d)|\\
&= \frac{1}{\log x} (\log |d| +O(1) )
+\log\log x
-\frac{1}{2}\sum_{\gamma_d \neq 0} 
\log \left(1 + \frac{ 1 }{ (\gamma_d \log x)^2} \right).
  \end{split}
\end{align}

On the other hand, similar to  \cite[Eq. (11)]{RaSo} and \cite[Lemma 14]{So} (\`a la Selberg),  for $\sigma\ge\frac{1}{2}$, one has
\begin{align*}
 -\frac{L'}{L}(\sigma,f\otimes\chi_d)
 &=\sum_{n\le x} \frac{\Lambda_f(n) \chi_d(n)}{n^{\sigma} } \frac{\log (x/n)}{\log x}
-\frac{1}{\log x}\left(\frac{L'}{L}\right)'(\sigma,f\otimes\chi_d)\\
&+\frac{1}{\log x}\sum_{\rho_d =\frac{1}{2}+i\gamma_d}  \frac{x^{\rho_d-\sigma}}{(\rho_d-\sigma)^2} 
+O\left( \frac{1}{x^{\sigma} \log x}\right), 
\end{align*}
for $x\ge 3$. If $L(\sigma_0 ,f\otimes \chi_d)\neq 0$, integrating the both sides from $\sigma_0$ to $\infty$, we obtain
\begin{align}\label{log-L-sigma0}
 \begin{split}
\log |L(\sigma_0 ,f\otimes \chi_d)|
& =
\sum_{n\le x} \frac{\Lambda_f(n) \chi_d(n)}{n^{\sigma_0} \log n} \frac{\log (x/n)}{\log x}
-\frac{1}{\log x}\Re \frac{L'}{L}(\sigma_0,f\otimes \chi_d)\\
&+\frac{1}{\log x}\sum_{\rho_d =\frac{1}{2}+i\gamma_d} \Re \int_{\sigma_0}^\infty \frac{x^{\rho_d-\sigma}}{(\rho_d-\sigma)^2}d\sigma 
+O\left( \frac{1}{ \sqrt{x} (\log x)^2}\right).
  \end{split}
\end{align}

For the first sum, the terms $n= p^k$ with $k\ge 3$ contribute at most $O(1)$. In addition, by the mean value theorem, there is $\xi_0\in (1,2\sigma_0)$ such that
$$
\sum_{p^2\le x} \frac{\Lambda_f(p^2) \chi_d(p^2)}{p \log p^2} \frac{\log (x/p^2)}{\log x}-
\sum_{p^2\le x} \frac{\Lambda_f(p^2) \chi_d(p^2)}{p^{1 +\frac{2}{\log x}} \log p^2} \frac{\log (x/p^2)}{\log x}
= \frac{1}{\log x}
\sum_{p^2\le x} \frac{\Lambda_f(p^2) \chi_d(p^2)}{p^{\xi_0} } \frac{\log (x/p^2)}{\log x},
$$
which is
$
\ll  \frac{1}{\log x}\sum_{p\le x} \frac{\log p}{p} \ll 1.
$
Therefore, by the estimate
$$
\sum_{ \substack{\p\le y\\ (p,N_0)=1}} \frac{\Lambda_f(p^2)}{p} 
=\sum_{ \substack{\p\le y\\ (p,N_0)=1}} \frac{(\lambda_f(p)^2-2)\log p }{p} 
=  -\log y +O(1),
$$
(which a direct consequence of the Rankin-Selberg theory; see, e.g., \cite[Ch. 5]{IK}),  and Abel's summation, the contribution of $p^2$ to the sum in \eqref{log-L-sigma0} is
$
-\oh \log\log x +O(1).
$
Secondly, it follows from \cite[Theorem 5.17]{IK} that
$$
\frac{L'}{L}(\sigma_0,f\otimes\chi_d)
\ll \frac{1}{2\sigma_0 -1}  (\log (k \sqrt{N}|d|))^{2-2\sigma_0  }  + \log\log (k \sqrt{N}|d|)
\ll_{k,N} \log x +\log\log |d|
$$
under GRH.
Thirdly, for $|\gamma_d\log x |\ge 1$, we know
$$
 \int_{\sigma_0}^{\infty} \frac{x^{\rho_d -\sigma}}{(\rho_d -\sigma)^2} d\sigma
\ll \frac{1}{\gamma_d^2}  \int_{\frac{1}{2}}^{\infty} x^{\frac{1}{2} -\sigma} d\sigma
\ll \frac{1}{\gamma_d^2 \log x},  
$$
which is $\le \log x$ for this range of $\gamma_d$. Also, in the two estimates above \cite[Eq. (15)]{RaSo}, for  $|\gamma_d\log x |\le 1$, one has
$$
\Re \int_{\frac{1}{2} +\frac{1}{\log x}}^{\infty} \frac{x^{\rho_d -\sigma}}{(\rho_d -\sigma)^2} d\sigma 
\ll   \int_{\frac{1}{2} +\frac{1}{\log x}}^{\infty} \frac{ x^{\frac{1}{2} -\sigma} }{(\frac{1}{2} -\sigma)^2} d\sigma \ll \log x.
$$

Thus, by \eqref{log-L-prime} and \eqref{log-L-sigma0}, (together with above discussion), recalling the definition \eqref{def-w}, 
we derive
\begin{align*}
\log | L'(\oh ,f\otimes \chi_d)|&=  \mathcal{P}_f(d; x) +\oh \log\log x +O(1)
+\frac{ \log |d|}{\log x} 
\\
&+ \frac{1}{\log x} (\log |d| +O(1) )
+\log\log x
+O\bigg(\sum_{\gamma_d \neq 0} 
\log \left(1 + \frac{ 1 }{ (\gamma_d \log x)^2} \right) \bigg),
\end{align*}
which concludes the proof.
\end{proof}

Now, by \eqref{def-SEd}, as $\Omega(E_d)\asymp 1/\sqrt{|d|}$, and $|E_d(\Bbb{Q})_{\mathrm{tors}}|$ is bounded, we can write
$$
\log |S(E_d)R(E_d)| = \log |L'(\oh, E_d)| + \log \sqrt{|d|} -\log \Tam(E_d) +O(1).
$$
Arguing as in \cite[pp. 1047-1048]{RaSo}, one can show that
\begin{align*}
&\sum_{\substack{ d\in\mathcal{F}\\ X/\log X \le |d| \le X}} 
|\log \Tam(E_d) +(\mu(E) +\oh)\log\log X -\mathcal{C}(d;x) |\\
& \ll X\log\log\log X +X \sum_{p< \log X} \frac{\log c(p)}{p}
+X \sum_{x<p\le  X} \frac{\log c(p)}{p}\\
&\ll X\log\log\log X +X \log\log\log\log X,
\end{align*}
where the last estimate follows from \eqref{est-cp}.
Thus, the number of $d \in\mathcal{F}$, with $X/\log X \le |d| \le X$, such that
\begin{equation}\label{bad-cond}
|\log \Tam(E_d) +(\mu(E) +\oh)\log\log X -\mathcal{C}(d;x) | \ge (\log\log\log X)^2
\end{equation}
is at most $\ll X/\log\log\log X$.
Therefore, the number of $d \in\mathcal{F}$ with $20 \le |d| \le X$ (and thus  $X \le |d| \le 2X$)  satisfying \eqref{bad-cond} is $\ll X/\log\log\log X$. Hence, upon dropping at most $\ll X/\log\log\log X$ $d\in \mathcal{F}$ with $X \le |d| \le 2X$, we have
$$
\log |S(E_d)R(E_d)| = \log |L'(\oh, E_d)| + \log \sqrt{|d|} +(\mu(E) +\oh)\log\log X -\mathcal{C}(d;x) +O((\log\log\log X)^2).
$$
This, together with Proposition \eqref{new-prop1}, yields
\begin{align}\label{key-expression-SEd}
 \begin{split}
&\log (|S(E_d)R(E_d)|/\sqrt{|d|} ) - (\mu(E) +1) \log\log |d|\\
& =   \mathcal{P}(d; x) - \mathcal{C}(d;x)
 +O\bigg( (\log\log\log X)^2 +
  \sum_{\gamma_{d}} \log \bigg(1+ \frac{1}{(\gamma_{d} \log x)^2} \bigg)  \bigg)
  \end{split}  
\end{align}
for all but at most $\ll X/\log\log\log X$  $d\in\mathcal{F}$ with $X\le |d|\le 2X$.

As may be noticed, the approximations \eqref{L'-exp} and \eqref{key-expression-SEd} are futile when there are zeros too close to $\frac{1}{2}$. We shall follow closely the idea of {Radziwi\l\l} and Soundararajan \cite[Lemma 1]{RaSo} to prove the following lemma in order to address this issue.

\begin{lemma}\label{lemma-G}
Let $\alpha_j<\beta_j$ be real numbers, and set $\underline{\alpha}= (\alpha_1,\ldots,\alpha_j)$ and $\underline{\beta}= (\beta_1,\ldots,\beta_j)$. Let $\mathcal{G}_{X}(\underline{\alpha},\underline{\beta};M)$ be the set of discriminants $d \in\mathcal{F}=\cap_{j=1}^{M}\mathcal{F}_{f_j}$, with  $X \le |d| \le 2X$, such that
$$
\mathcal{Q}_j(d;X) = \frac{\mathcal{P}_j(d; x)}{\sqrt{\log\log X}} \in (\alpha_j,\beta_j)\enspace \forall 1\le j \le M
$$
while each $ L(s,f_j\otimes\chi_d)$ has no zeros $\frac{1}{2}+ i\gamma_d$, with
$
0< |\gamma_d|  \le( (\log X) (\log\log X))^{-1},
$
and no double zeros at $s=\frac{1}{2}$. Then for any $\delta>0$, we have
$$
\mathcal{G}_{X}(\underline{\alpha},\underline{\beta};M)
\ge
\left(C_M-\delta\right)
( \Psi_M(\underline{\alpha},\underline{\beta}) +o(1) )
\# \{ d \in\mathcal{F} :  X \le |d| \le 2X \},
$$
where $C_M$ and $\Psi_M(\underline{\alpha},\underline{\beta})$ are defined as in \eqref{def-Psi_M}.
\end{lemma}

\begin{proof}
Let $\Phi$ be a smooth approximation to the indicator function of $[1, 2]$,
and let $\kappa$ and $a \enspace \mymod{N_0}$ be as in Section 2. It follows from \eqref{M-P-moments} and the method of moments (as in \cite{FS}) that for any $a_j\in \Bbb{R}$, we know that $\sum_{j=1}^M a_j \mathcal{Q}_j(d;X)$ converges to a normal random variable $\mathcal{Z}_{a_1,\ldots,  a_M}$, with mean 0 and variance $ \sum_{j=1}^M a_j^2$, in distribution, whose characteristic function is  
\begin{align}\label{char function trick}
\phi_{\mathcal{Z}_{a_1,\ldots, a_M}}(u) = \exp\left(-\oh \left( a_1^2 + \cdots + a_M^2 \right) u^{2}\right).
\end{align}
Furthermore, as the identity $M\times M$ matrix $I_M$ is positive definite, there exists a multivariate normal distribution $ (\mathcal{Q}_{1},\ldots, \mathcal{Q}_{M})$, with $\mathcal{Q}_{j}$ being standard normal, and the correlation between $\mathcal{Q}_{i}$ and  $\mathcal{Q}_{j}$ is $0$ when $i\neq j$ (see \cite[Theorem 2.1]{HH}). Hence, by \cite[Theorem 5.5.33]{Dudewicz}, a direct computation shows that   $\sum_{j=1}^M a_j\mathcal{Q}_{j}$ is a normal distribution, for any $a_j\in\Bbb{R}$, such that the characteristic function of $\sum_{j=1}^M a_j\mathcal{Q}_{j}$is the same as the characteristic function of $\mathcal{Z}_{a_1,\ldots,a_M}$ given in \eqref{char function trick}. Since the characteristic function uniquely determines the distribution (see \cite[Theorem 9.5.1]{Resnick}), we deduce that $\mathcal{Z}_{a_1,\ldots,a_M} = \sum_{j=1}^M a_j\mathcal{Q}_{j}$ 
in distribution, and thus  $\sum_{j=1}^M a_j\mathcal{Q}_{j}(d;X)$  converges to $\sum_{j=1}^M a_j\mathcal{Q}_{j}$ in distribution  for any real $a_j$. Therefore, the Cram\'er-Wold device implies that $(\mathcal{Q}_{1}(d;X),\ldots, \mathcal{Q}_{M}(d;X))$ converges to $(\mathcal{Q}_{1},\ldots, \mathcal{Q}_{M})$, in distribution, whose joint probability distribution is given by
$
\frac{1}{2\pi \sqrt{\text{det}(I_M)}}e^{-\frac{1}{2}{\bf v}^{\mathrm{T}}I_M^{-1}{\bf v}}.
$
To summarise, we have established that
\begin{equation}\label{1st-asym}
\sum_{\substack{d\in\mathcal{F}(\kappa,a)\\ \mathcal{Q}_i(d;X) \in (\alpha_i,\beta_i)\,\forall i }}  \Phi\left(\frac{\kappa d}{X}\right)
= ( \Psi_M(\underline{\alpha},\underline{\beta})  +o(1) ) \sum_{d\in\mathcal{F}(\kappa,a)} \Phi\left(\frac{\kappa d}{X}\right).
\end{equation}

 Similarly, by \eqref{M-P-h-moments},  choosing $h$ to be the Fej\'er kernel and $L = (2-\eta ) \log X$,  one can show that
\begin{align}\label{2nd-asym}
\begin{split}
\sum_{\substack{d\in\mathcal{F}(\kappa,a)\\ \mathcal{Q}_i(d;X) \in (\alpha_i,\beta_i)\,\forall i  }}  \sum_{\gamma_{d}} h\left( \frac{\gamma_{d} L}{2\pi}\right)  \Phi\left(\frac{\kappa d}{X}\right)
&= ( \Psi_E(\underline{\alpha},\underline{\beta})   +o(1) ) \sum_{d\in\mathcal{F}(\kappa,a)}  \sum_{\gamma_{d}} h\left( \frac{\gamma_{d} L}{2\pi}\right)  \Phi\left(\frac{\kappa d}{X}\right)\\
&=  M  ( \Psi_E(\underline{\alpha},\underline{\beta})   +o(1) )
 \left(\frac{1}{1-\frac{\eta}{2} } +\frac{1}{2} +o(1)\right)
 \sum_{d\in\mathcal{F}(\kappa,a)}    \Phi\left(\frac{\kappa d}{X}\right),
 \end{split}
\end{align}
where $\gamma_d$ are running over the imaginary part of the (non-trivial) zeros of $\prod_{j=1}^M L(s,f_j\otimes\chi_d)$. 

Note that the ``weight'' $\sum_{\gamma_{d}} h ( \frac{\gamma_{d} L}{2\pi} )$ is non-negative. Also, if $L(s,f_j\otimes\chi_d)$ has a zero $\frac{1}{2} +i\gamma_d$ with $0<|\gamma_{d}| \le ((\log X) (\log \log X))^{-1}$, then there exists a complex conjugate pair of such zeros. In addition, if $L'(\frac{1}{2},f_j\otimes\chi_d)=0$, then $L(s,f_j\otimes\chi_d)$ has triple zeros at $s=\frac{1}{2}$ as $\Lambda''(s,f_j\otimes\chi_d) = -\Lambda''(1-s,f_j\otimes\chi_d)$ for $d\in \mathcal{F}$. 

To proceed further, we let $\mathcal{Z}$ stand for the set of fundamental discriminants $d\in\mathcal{F}$ such that every $L(s,f_j\otimes \chi_d)$ has no zeros with  $0<|\gamma_{d}| \le ((\log X) (\log \log X))^{-1}$ and no double zeros at $s=\frac{1}{2}$. From \eqref{1st-asym} and \eqref{2nd-asym},
it then follows that
$$
  M \left(\frac{1}{1+\frac{\eta}{2} } +\frac{1}{2} +o(1)\right) 
 \sum_{\substack{d\in\mathcal{F}(\kappa,a)\\ \mathcal{Q}_i(d;X) \in (\alpha_i,\beta_i)\,\forall i }} \Phi\left(\frac{\kappa d}{X}\right)
   =\sum_{\substack{d\in\mathcal{F}(\kappa,a) 
  \\ \mathcal{Q}_i(d;X) \in (\alpha_i,\beta_i)\,\forall i  }}
   \sum_{\gamma_{d}} h\left( \frac{\gamma_{d} L}{2\pi}\right)  \Phi\left(\frac{\kappa d}{X}\right)
$$
can be written as
$$
\sum_{\substack{d\in\mathcal{F}(\kappa,a) 
  \\ \mathcal{Q}_i(d;X) \in (\alpha_i,\beta_i)\,\forall i  }}
 \Phi\left(\frac{\kappa d}{X}\right)
   +
\sum_{\substack{d\in\mathcal{F}(\kappa,a) 
  \\ \mathcal{Q}_i(d;X) \in (\alpha_i,\beta_i)\,\forall i  }}
   \sideset{}{^*}\sum_{\gamma_{d}} h\left( \frac{\gamma_{d} L}{2\pi}\right)  \Phi\left(\frac{\kappa d}{X}\right),
$$
where the stared sum is over zeros of $\prod_{j=1}^M L(s,f_j\otimes\chi_d)$ with multiplicity of the zeros at $s=\frac{1}{2}$ subtracted by $M$. Therefore, we derive
\begin{align*}
& \left(\frac{  M }{1-\frac{\eta}{2} } +\frac{  M }{2} - M  +o(1)\right) 
 \sum_{\substack{d\in\mathcal{F}(\kappa,a)\\ \mathcal{Q}_i(d;X) \in (\alpha_i,\beta_i)\,\forall i }} \Phi\left(\frac{\kappa d}{X}\right) \\
& \ge  0+ 
2   \sum_{\substack{d\in\mathcal{F}(\kappa,a)\backslash  \mathcal{Z} \\ \mathcal{Q}_i(d;X) \in (\alpha_i,\beta_i)\,\forall i }}  \Phi\left(\frac{\kappa d}{X}\right)
 =  2 \sum_{\substack{d\in\mathcal{F}(\kappa,a)  \\ \mathcal{Q}_i(d;X) \in (\alpha_i,\beta_i)\,\forall i  }}  \Phi\left(\frac{\kappa d}{X}\right)
- 2  \sum_{\substack{d\in\mathcal{F}(\kappa,a) \cap  \mathcal{Z} \\\mathcal{Q}_i(d;X) \in (\alpha_i,\beta_i)\,\forall i  }}  \Phi\left(\frac{\kappa d}{X}\right)
\end{align*}
as $\mathcal{F}(\kappa,a)$ is the disjoint union of $\mathcal{F}(\kappa,a)\cap  \mathcal{Z}$ and  $\mathcal{F}(\kappa,a)\backslash  \mathcal{Z} $. 
Therefore, we arrive at
\begin{align*}
  \sum_{\substack{d\in\mathcal{F}(\kappa,a) \cap  \mathcal{Z} \\ \mathcal{Q}_i(d;X) \in (\alpha_i,\beta_i)\,\forall i }}  \Phi\left(\frac{\kappa d}{X}\right)
\ge  
\left(1 - \frac{  M }{2-\eta } +\frac{  M }{4}  +o(1)\right)  \sum_{\substack{d\in\mathcal{F}(\kappa,a)  \\ \mathcal{Q}_i(d;X) \in (\alpha_i,\beta_i)\,\forall i }}   \Phi\left(\frac{\kappa d}{X}\right),
\end{align*}
which completes the proof upon summing over all possible pairs $(\kappa,a)$.
\end{proof}


We also have the following lower bound for the joint distribution of $\mathcal{P}(d; x)$ and $\mathcal{P}(d; x)-\mathcal{C}(d;x)$.

\begin{lemma}\label{lemma-H}
Let $E/\Bbb{Q}$ be an elliptic curve. Let $\mathcal{P}(d; x)$, $\mathcal{C}(d; x)$, and $\mathcal{F}$ be as introduced above Proposition \ref{P-C-moments}. Let $\alpha_j<\beta_j$ be real numbers, and set $\underline{\alpha}= (\alpha_1,\alpha_2)$ and $\underline{\beta}= (\beta_1,\beta_2)$.  Define $\mathcal{H}_X(\underline{\alpha},\underline{\beta})$ to be the set of discriminants $d \in\mathcal{F}$, with  $X \le |d| \le 2X$, such that
$$
\mathcal{R}_1(d;X) = \frac{\mathcal{P}(d; x)}{\sqrt{\log\log X}} \in (\alpha_1,\beta_1) \quad\text{and}\quad
\mathcal{R}_2(d;X)=\frac{\mathcal{P}(d; x)-\mathcal{C}(d;x)}{\sqrt{\sigma(E)^2 \log\log X}} \in (\alpha_2,\beta_2), 
$$
while each $ L(s,f_j\otimes\chi_d)$ has no zeros $\frac{1}{2}+i\gamma_d$, with
$
0< |\gamma_d|  \le( (\log X) (\log\log X))^{-1},
$
and no double zeros at $s=\frac{1}{2}$. 
Then for any $\delta>0$, we have
$$
\mathcal{H}_X(\alpha,\beta) 
\ge
\left(\frac{1}{4}-\delta\right)
( \Xi_E(\underline{\alpha},\underline{\beta}) +o(1) )
\# \{ d \in\mathcal{F} :  X \le |d| \le 2X \},
$$
where $\Xi_E(\underline{\alpha},\underline{\beta})$ is defined as in \eqref{def-XiE-KE}.
\end{lemma}

\begin{proof}
As the proof follows the same lines of the one of Lemma \ref{lemma-G}, we only emphasise here the main difference. It follows from \eqref{P-C-moments-1}  and the method of moment that for any $a_1,a_2\in \Bbb{R}$, $a_1 \mathcal{R}_1(d;X) +  a_2 \mathcal{R}_2(d;X)$  converges to a normal random variable $\mathcal{Z}_{a_1,a_2}$, with mean 0 and variance $ a_1^2 + 2a_1a_2 \sigma(E)^{-1} + a_2^2$, in distribution. In addition, the characteristic function of $\mathcal{Z}_{a_1,a_2}$ is 
$$
\phi_{\mathcal{Z}_{a_1,a_2}}(u) = \exp\left(-\oh \left( a_1^2 + 2 a_1a_2\sigma(E)^{-1} + a_2^2 \right) u^{2}\right).
$$

On the other hand, as  $0< \sigma(E)^{-1}<1$,  Sylvester's criterion (see, e.g., \cite{Sylvester's criterion}) tells us that the matrix
\begin{equation}\label{cov-mat-E}
\mathfrak{K}_E 
=\begin{pmatrix} 
1 & \sigma(E)^{-1}  \\
\sigma(E)^{-1} & 1  \\
\end{pmatrix} 
\end{equation}
is always positive definite. Hence, by \cite[Theorem 2.1]{HH},  there is a bivariate normal distribution $ (\mathcal{R}_{1}, \mathcal{R}_{2})$ such that each $\mathcal{R}_{i}$ is standard normal, and the correlation between $\mathcal{R}_{1}$ and  $\mathcal{R}_{2}$ equals $\sigma(E)^{-1}$. Therefore, for any $a_1,a_2\in\Bbb{R}$, $a_1\mathcal{R}_{1}+a_2 \mathcal{R}_{2}$ is a normal distribution. In addition, as $\Var(\mathcal{Q}_i)=1$ for each $i$, we have $\Cov(\mathcal{R}_{1}, \mathcal{R}_{2})= \sigma(E)^{-1}$, and thus 
\begin{align*}
\Var(a_1\mathcal{R}_{1}+a_2 \mathcal{R}_{2}) 
= a^{2}_{1}\Var (\mathcal{R}_{1}) + a^{2}_{2}\Var (\mathcal{R}_{2}) + 2a_{1}a_{2}\Cov(\mathcal{R}_{1}, \mathcal{R}_{2})
=  a_1^2 + 2 a_1a_2\sigma(E)^{-1} + a_2^2.
\end{align*}
From which, we verify that the characteristic function of $a_1\mathcal{R}_{1}+a_2 \mathcal{R}_{2}$ is the same as the one of $\mathcal{Z}_{a_1,a_2}$, so we derive that
$
\mathcal{Z}_{a_1,a_2} = a_1\mathcal{R}_{1}+a_2 \mathcal{R}_{2}
$
in distribution. Consequently, $a_1\mathcal{R}_{1}(d;X)+a_2 \mathcal{R}_{2}(d;X)$  converges to $a_1\mathcal{R}_{1}+a_2 \mathcal{R}_{2}$ in distribution  for any real $a_i$. Consequently, $(\mathcal{R}_{1}(d;X), \mathcal{R}_{2}(d;X))$  converges to $(\mathcal{R}_{1}, \mathcal{R}_{2})$, in distribution, whose joint probability distribution is 
$
\frac{1}{2\pi \sqrt{\text{det}(\mathfrak{K}_E)}}e^{-\frac{1}{2}{\bf v}^{\mathrm{T}}\mathfrak{K}_E^{-1}{\bf v}},
$
with $\mathfrak{K}_E$ defined in  \eqref{cov-mat-E}. In other words, we have shown that
$$
\sum_{\substack{d\in\mathcal{F}(\kappa,a)\\ \mathcal{R}_i(d;X) \in (\alpha_i,\beta_i)\,\forall i }}  \Phi\left(\frac{\kappa d}{X}\right)
= ( \Xi_E(\underline{\alpha},\underline{\beta})  +o(1) ) \sum_{d\in\mathcal{F}(\kappa,a)} \Phi\left(\frac{\kappa d}{X}\right).
$$
Moreover, by \eqref{P-C-h-moments} and an analogous argument as above, we have
\begin{align*}
\begin{split}
\sum_{\substack{d\in\mathcal{F}(\kappa,a)\\ \mathcal{R}_i(d;X) \in (\alpha_i,\beta_i)\,\forall i  }}  \sum_{\gamma_{d}} h\left( \frac{\gamma_{d} L}{2\pi}\right)  \Phi\left(\frac{\kappa d}{X}\right)
&= ( \Xi_E(\underline{\alpha},\underline{\beta})   +o(1) ) \sum_{d\in\mathcal{F}(\kappa,a)}  \sum_{\gamma_{d}} h\left( \frac{\gamma_{d} L}{2\pi}\right)  \Phi\left(\frac{\kappa d}{X}\right).
 \end{split}
\end{align*}
Hence, proceeding  with the same line of reasoning of proving Lemma \ref{lemma-G}, we conclude the proof.
\end{proof}

In addition, the proof of \cite[Lemma 2]{RaSo}, combined with Proposition \ref{M-forms-moments} (instead of \cite[Proposition 2]{RaSo}), yields the following lemma that allows us to control the contribution of zeros away from $\frac{1}{2}$.

\begin{lemma}\label{RS-lemma2}
The number of discriminants $d \in\mathcal{F}$, with $X \le |d| \le 2X$, such that
$$
\sum_{ |\gamma_{d}|  \ge( (\log X) (\log\log X))^{-1}  }     \log \bigg(1+ \frac{1}{(\gamma_{d} \log x)^2} \bigg)  
\ge  (\log \log\log X)^3,
$$
is $\ll_M X/\log \log\log X$. Here, the sum is over the zeros of $\prod_{j=1}^M L(s,f_j \otimes \chi_d)$.
\end{lemma}

Finally, we are in a position to prove Theorem \ref{main-thm-M-forms}. (We shall, however, omit the proof of Theorem \ref{main-thm} since it is almost the same upon using \eqref{key-expression-SEd} and  Lemma \ref{lemma-H} instead.)

\begin{proof}[Proof of Theorem \ref{main-thm-M-forms}]
By Lemma \ref{lemma-G},  for $d\in \mathcal{G}_X(\underline{\alpha},\underline{\beta})$, we have
$
\mathcal{Q}_j(d;X) = \frac{\mathcal{P}_j(d; x)}{\sqrt{\log\log X}} \in (\alpha_j,\beta_j)$ for all $1\le j \le M
$
while each $ L(s,f_j\otimes\chi_d)$ has no zeros $\frac{1}{2}+\gamma_d$ with
$
 |\gamma_d|  \le( (\log X) (\log\log X))^{-1}.
$
Also, (for each $j$), Lemma \ref{RS-lemma2} allows us to remove $\ll  X/\log \log\log X$ $d$ from  $\mathcal{G}_X(\underline{\alpha},\underline{\beta})$ so that the
contribution of zeros  of $L(s,f_j\otimes \chi_d)$ with $|\gamma_{d}| \ge ((\log X) (\log\log X))^{-1} $ to the last sum in \eqref{L'-exp} with $f=f_j$ is bounded by $ (\log \log\log X)^3$. Hence, there are at least $$
\left(C_M-\delta\right)
( \Psi_M(\underline{\alpha},\underline{\beta}) +o(1) )
\# \{ d \in\mathcal{F} :  X \le |d| \le 2X \}
$$
fundamental discriminants $d\in\mathcal{F}$ with $X\le |d|\le 2X$ such that
$$
\frac{ \log | L'(\oh ,f_j\otimes \chi_d)| - \frac{1}{2}\log \log X }{\sqrt{ \log\log X}} = \frac{\mathcal{P}_j(d; x) }{\sqrt{ \log\log X}}
 +O\left(\frac{ (\log \log\log X)^3}{\sqrt{\log\log X}}\right)\in (\alpha_j,\beta_j)\enspace \forall 1\le j \le M,
$$
which yields the claimed lower bound. 
\end{proof}

\section{Proof of Proposition  \ref{key-prop-2}}\label{proof-key}

We start with the  explicit formula
\begin{align*}
\sum_{\gamma_{d}} h\left( \frac{\gamma_{d}}{2\pi}\right)
&=\frac{1}{2\pi}\int_{-\infty}^\infty h\left( \frac{t}{2\pi}\right)
\bigg(  \log \prod_{j=1}^M \frac{N_jd^2}{(2\pi)^2}  ++\sum_{j=1}^M \frac{\Gamma'}{\Gamma}(\frac{k_j}{2}+it) +  \frac{\Gamma'}{\Gamma}(\frac{k_j}{2}-it)\bigg) dt\\
&- \sum_{n=1}^\infty \frac{\Lambda_\mathcal{L}(n)}{\sqrt{n}} \chi_d(n)
\left(\widehat{h} (\log n ) +\widehat{h} (\log n) \right),
\end{align*}
where the first sum is over the zeros of $\prod_{j=1}^M L(s, f_j\otimes \chi_d)$,  and
$
\Lambda_\mathcal{L}(n)  =  \sum_{j=1}^M  \Lambda_{f_j}(n).
$
(Here, $\Lambda_{f_j}(n)$ are the coefficients of 
$
\frac{L'}{L} (s,f_j)=  \sum_{n=1}^\infty \Lambda_{f_j}(n)n^{-s}
$
such that $|\Lambda_{f_j}(n)|\le 2\Lambda(n)$; particularly,  $\Lambda_{f_j}(n)$ is only supported on prime powers.) Using this explicit formula (with dilated function $h_L(x) = h(xL)$ so that $\widehat{h_L}(\xi) = \frac{1}{L} \widehat{h}(\frac{\xi}{L})$, we obtain
\begin{align*}
\sum_{\gamma_{d}} h\left( \frac{\gamma_{d} L}{2\pi}\right)
&=\frac{1}{2\pi}\int_{-\infty}^\infty h\left( \frac{t L}{2\pi}\right)
\bigg(  \log \prod_{j=1}^M \frac{N_jd^2}{(2\pi)^2}  +\sum_{j=1}^M \frac{\Gamma'}{\Gamma}(\frac{k_j}{2}+it) +  \frac{\Gamma'}{\Gamma}(\frac{k_j}{2}-it)\bigg) dt\\
&-\frac{1}{L} \sum_{n=1}^\infty \frac{\Lambda_\mathcal{L}(n)}{\sqrt{n}} \chi_d(n)
\left(\widehat{h} \left(\frac{\log n}{L} \right) +\widehat{h} \left(-\frac{\log n}{L} \right) \right).
\end{align*}
This allows us to write
$
\mathcal{S}_{\kappa,a, v}  =S_1 -S_2,
$
where
$$
S_1 = \frac{1}{2\pi}\int_{-\infty}^\infty h\left( \frac{t L}{2\pi}\right)
\sum_{\substack{ d\in\mathcal{F}(\kappa,a)\\ v\mid d}}
\bigg(  \log \prod_{j=1}^M \frac{N_jd^2}{(2\pi)^2}  +\sum_{j=1}^M \frac{\Gamma'}{\Gamma}(\frac{k_j}{2}+it) +  \frac{\Gamma'}{\Gamma}(\frac{k_j}{2}-it)\bigg) \chi_d(\ell) \Phi\left(\frac{\kappa d}{X}\right) dt
$$
and
$$
S_2 
=  \frac{1}{L} \sum_{n=1}^\infty \frac{\Lambda_\mathcal{L}(n)}{\sqrt{n}} 
\left(\widehat{h} \left(\frac{\log n}{L} \right) +\widehat{h} \left(-\frac{\log n}{L} \right) \right) \sum_{\substack{ d\in\mathcal{F}(\kappa,a)\\ v\mid d}} \chi_d(\ell n) \Phi\left(\frac{\kappa d}{X}\right).
$$

Similar to \cite{RaSo},  by \eqref{key-prop-RS-original}, it can be shown that
\begin{equation*}
S_1
= \delta(\ell=\square)
\frac{MX}{vN_0} \prod_{p\mid \ell v} \left(1 + \frac{1}{p} \right)^{-1}
 \prod_{p\nmid N_0 } \left(1 - \frac{1}{p^2} \right) \widehat{\Phi}(0)
 (2\log X +O(1))\frac{\widehat{h}(0)}{L}
 + O(X^{\frac{1}{2}+\varepsilon}\ell^{\frac{1}{2}}),
\end{equation*}
where the $O(1)$-term depends on the weights $k_j$ of $f_j$. (Indeed, the evaluation for $S_1$ only differs from \cite{RaSo} in the appearance of $k_j$. More precisely, as $h$ is even, we have
\begin{align*}
\int_{-\infty}^{\infty}\frac{\Gamma'}{\Gamma}(\frac{k_j}{2}\pm it)  h\left( \frac{t L}{2\pi}\right)dt
&=\frac{2\pi}{L}\int_{-\infty}^{\infty}\frac{\Gamma'}{\Gamma}\left(\frac{k_j}{2}\pm i \frac{ 2\pi t}{L}\right) h(t)dt \\
&=\frac{4\pi}{L}\int_{0}^{\infty}\left(\frac{\Gamma'}{\Gamma}\left(\frac{k_j}{2}+i\frac{2\pi t}{L}\right)+\frac{\Gamma'}{\Gamma}\left(\frac{k_j}{2}-i\frac{2\pi t}{L}\right)\right)h(t)dt,
\end{align*}
which is
$$
\int_{0}^{\infty}\left( 2\frac{\Gamma'}{\Gamma}\left(\frac{k_j}{2}\right)
+O\left(\frac{ t^2}{ (k_j L)^2}\right)\right)h(t) dt 
  \ll_{k_j} \widehat{h}(0) 
$$
since $ 
\frac{\Gamma'}{\Gamma}(a+bi)+\frac{\Gamma'}{\Gamma}(a-bi)=2\frac{\Gamma'}{\Gamma}(a)+O(a^{-2}b^2)$ for  $a>0$ and $b\in\mathbb{R}$.)


Now, we turn our attention to $S_2$. We start by noting that as $v\mid d$, $\chi_d(\ell n) = 0$ if $(n,v)>1$. Hence, we may assume $(n,v)=1$ throughout our argument (especially, for $n$ appearing in  $S_2$). Moreover, as argued in the paragraph leading to \cite[Eq. (22)]{RaSo}, since $d$ is fixed in
a residue class $\mymod{N_0}$, if $n$ is a prime power dividing $N_0$, then $\chi_d(n)$ is determined by the congruence condition on $d$. Consequently, as  $v\sqrt{\ell}\le X^{\frac{1}{2}-2\varepsilon}$ by our assumption, it follows from \eqref{key-prop-RS-original} that adding the condition $(n,N_0)=1$ to the involving sum in $S_2$ contributes an error at most
$$
\ll\frac{1}{L} \sum_{(n,N_0)>1}  \frac{M\Lambda(n)}{\sqrt{n}} 
\bigg|  \sum_{\substack{ d\in\mathcal{F}(\kappa,a)\\ v\mid d}} \chi_d(\ell) \Phi\left(\frac{\kappa d}{X}\right) \bigg|
\ll \delta(\ell=\square) \frac{MX}{vL} + MX^{\frac{1}{2} +\varepsilon} \ell^{\frac{1}{2}}.
$$

Recall that $d\in \mathcal{F}(\kappa,a)$ has to be square-free. So, for $d$ coprime to $N_0$ such that $v\mid d$, since $v$ is also square-free and coprime to $N_0$, one has 
\begin{equation*}
\sum_{\beta\mid (v,d/v)}\mu (\beta) \sum_{\substack{ (\alpha , vN_0)=1\\ \alpha^2\mid d/v }}\mu (\alpha)
= \begin{cases}
1 & \text{if $d$ is square-free and $v\mid d$;}  \\
 0 & \text{otherwise}.
   \end{cases}
\end{equation*}
From this, as in \cite[Eq. (20)]{RaSo15}, writing $d = k v\beta \alpha^2$, one then obtains
$$
\sum_{\substack{ d\in\mathcal{F}(\kappa,a)\\ v\mid d}}  \chi_d(\ell n) \Phi\left(\frac{\kappa d}{X}\right)
= 
\sum_{\beta\mid v} \sum_{ (\alpha ,\ell n vN_0)=1} \mu (\beta)\mu (\alpha) \left(\frac{v\beta \alpha^2}{\ell n}\right)
  \sum_{ k\equiv a \overline{ v\beta \alpha^2  } \mymod{N_0} }  \left(\frac{k}{\ell n}\right)\Phi\left(\frac{\kappa k v\beta \alpha^2}{X}\right).
$$
For a positive parameter  $A\le X$ to be chosen later, bounding the sum over $k$ trivially, one can see that the contribution  of the terms in the expression with  $\alpha > A$ is
$$
\ll\sum_{\beta\mid v} \sum_{ \alpha > A}  \frac{X}{v\beta \alpha^2} \ll\frac{Xv^{\varepsilon}}{vA}.
$$
Therefore, under GRH, we know the terms with $\alpha>A$ in $S_2$ the contributes an error at most
\begin{equation}\label{1st-error}
\ll \frac{1}{L} \sum_{n=1}^\infty \frac{\Lambda_\mathcal{L}(n)}{\sqrt{n}} 
\left(\widehat{h} \left(\frac{\log n}{L} \right) +\widehat{h} \left(-\frac{\log n}{L} \right) \right)  \frac{Xv^{\varepsilon}}{vA}
\ll \frac{MX L v^\varepsilon}{vA}\ll \frac{MX\log X}{v^{1-\varepsilon} A }.
\end{equation}

To handle remaining terms in $S_2$, setting
$$
H(\xi) =   \widehat{h}(\xi)  +\widehat{h} (-\xi) ,
$$
we can write the terms with $\alpha\le A$ in $S_2$ as
\begin{align}\label{alpha<A}
\begin{split}
 \frac{1}{L} \sum_{n=1}^\infty \frac{\Lambda_\mathcal{L}(n)}{\sqrt{n}} 
H\left(\frac{\log n}{L} \right) 
\sum_{\beta\mid v} \sum_{ \substack{(\alpha ,\ell n vN_0)=1 \\ \alpha\le A }} \mu (\beta)\mu (\alpha) \left(\frac{v\beta \alpha^2}{\ell n}\right)
  \sum_{ k\equiv a \overline{ v\beta \alpha^2  } \mymod{N_0} }  \left(\frac{k}{\ell n}\right)\Phi\left(\frac{\kappa k v\beta \alpha^2}{X}\right).
  \end{split}
\end{align}
Applying the Poisson summation formulated in \cite[Lemma 7]{RaSo15}, one can express the inner sum over $k$ in \eqref{alpha<A} as
\begin{equation}\label{inner-expression}
\frac{X}{\ell n N_0 v\beta\alpha^2}   \left(\frac{\kappa N_0}{\ell n}\right)
\sum_{m } \widehat{\Phi}\left( \frac{Xm }{\ell  n N_0 v\beta\alpha^2}\right)
e\bigg( \frac{ma \overline{ v\beta\alpha^2 \ell  n}}{N_0} \bigg) \tau_m (\ell n),
\end{equation}
where $\tau_m (\ell n)$ is a Gauss sum defined by
$$
\tau_m (\ell n) = \sum_{ b \mymod{\ell n} }  \left(\frac{b}{\ell n}\right) e\left(\frac{mb}{\ell n}\right)
 = \left( \frac{1+i}{2} + \left(\frac{-1}{\ell n}\right) \frac{1-i}{2}  \right)G_m (\ell n),
$$
and
$$
G_m (\ell n) = \left( \frac{1-i}{2} + \left(\frac{-1}{\ell n}\right) \frac{1+i}{2}  \right)\sum_{ b \mymod{\ell n} }  \left(\frac{b}{\ell n}\right) e\left(\frac{mb}{\ell n}\right).
$$

We first consider the terms with $m\neq 0$, and note that we may assume $n\le e^L$ as  $\widehat{h}$ is compactly supported in $[-1, 1]$ (and so is $H$). In addition, since $\widehat{\Phi}(s) \ll_K \frac{1}{|s|^K}$ for any $K>0$, the terms with $|m|> B$ in \eqref{inner-expression} contributes at most
 \begin{equation*}
\ll_K \frac{X}{\ell n N_0 v\beta\alpha^2}  
\sum_{|m|>B } \left( \frac{\ell  n N_0 v\beta\alpha^2} {Xm }\right)^K
(\ell n)
\ll \ell  n   \left( \frac{\ell  n N_0 v\beta\alpha^2} {X }\right)^{K-1} \frac{1}{B^{K-1}}
\ll  \ell  n X^{- (K-1)\varepsilon}
\end{equation*}
where we used the trivial bound  $|\tau_m (\ell n)| \le \ell n$, provided that $ B\ge \ell e^L A^2 X^{-1+ 3\varepsilon}$ and $K> 1$. Hence, choosing $K$ to be sufficiently large (depending on $\varepsilon>0$), the contribution of the terms with $|m|> \ell e^L A^2 X^{-1+ 3\varepsilon}$ from \eqref{inner-expression} to $S_2$ is  $\ll 1$.

So, remaining terms (with  $0<|m| \le \ell e^L A^2 X^{-1+ 3\varepsilon}$) in $S_2$ becomes
\begin{align*}
\begin{split}
&  \frac{1}{L} \sum_{n=1}^\infty \frac{\Lambda_\mathcal{L}(n)}{\sqrt{n}} 
H\left(\frac{\log n}{L} \right)\sum_{\beta\mid v} \sum_{ \substack{(\alpha ,\ell n vN_0)=1 \\ \alpha\le A }} \mu (\beta)\mu (\alpha) \left(\frac{v\beta \alpha^2}{\ell n}\right)\\
&\times 
\frac{X}{\ell n N_0 v\beta\alpha^2}   \left(\frac{\kappa N_0}{\ell n}\right)
\sum_{ 0<|m| \le \ell e^L A^2 X^{-1+ 3\varepsilon} } \widehat{\Phi}\left( \frac{Xm }{\ell  n N_0 v\beta\alpha^2}\right)
e\bigg( \frac{ma \overline{ v\beta\alpha^2 \ell  n}}{N_0} \bigg) \tau_m (\ell n).
  \end{split}
\end{align*}
Upon a rearrangement, it is
\begin{align}\label{rearr}
\begin{split}
&  \frac{X}{\ell L N_0} \sum_{ 0<|m| \le \ell e^L A^2 X^{-1+ 3\varepsilon} } \sum_{\beta\mid v} \sum_{ \substack{(\alpha ,\ell  vN_0)=1 \\ \alpha\le A }}   \frac{\mu (\beta)\mu (\alpha)}{  v\beta\alpha^2}  \sum_{(n ,\alpha N_0)=1}\frac{\Lambda_\mathcal{L}(n)}{n\sqrt{n}} 
H\left(\frac{\log n}{L} \right)\\
&\times 
  \left(\frac{\kappa v\beta N_0}{\ell n}\right)
\widehat{\Phi}\left( \frac{Xm }{\ell  n N_0 v\beta\alpha^2}\right)
e\bigg( \frac{ma \overline{ v\beta\alpha^2 \ell  n}}{N_0} \bigg) \tau_m (\ell n).
  \end{split}
\end{align}

In light of relation between $\tau_m (\ell n)$ and $G_m (\ell n)$, to handle the sum over $n$, 
 we shall bound
\begin{equation}\label{expression-in-G}
\sum_{(n ,\alpha v N_0)=1}\frac{\Lambda_\mathcal{L}(n)}{n\sqrt{n}} 
H\left(\frac{\log n}{L} \right)  \left(\frac{\pm \kappa v\beta N_0}{\ell n}\right)
\widehat{\Phi}\left( \frac{Xm }{\ell  n N_0 v\beta\alpha^2}\right)
e\bigg( \frac{ma \overline{ v\beta\alpha^2 \ell  n}}{N_0} \bigg) G_m (\ell n),
\end{equation}
and we therefore require the following explicit calculation for  $G_m (\ell n)$ from \cite[Lemma 2.3]{So00}.

\begin{lemma}
If $\ell$ and  $n$ are coprime and odd, then $G_m (\ell n) =G_m (\ell)G_m ( n)$. Moreover, if $p^\alpha$ is the largest prime power of $p$ that divides $m$ (when $m=0$, setting $\alpha=\infty$), then 
\begin{align*}
G_m (p^\beta)
=
\begin{cases}
\phi(p^\beta) &\mbox{if $\beta\le \alpha$ is even;}\\
-p^\alpha &\mbox{if $\beta= \alpha+1 $ is even;}\\
(\frac{mp^{-\alpha}}{p})  p^{\alpha + \frac{1}{2}}   &\mbox{if $\beta= \alpha+1 $  is odd;}\\
\end{cases}
\end{align*}
and $G_m (p^\beta) =0$ if $\beta\le \alpha$ is odd, or $\beta\ge \alpha+2$. 
\end{lemma}

As discussed in \cite[p. 12]{RaSo}, if $n$ is a prime power such that $(n,m)=1$, then $G_{m}(\ell n) =0$ unless $n$ is a prime $p \nmid \ell $ (for such in instance, $G_m(\ell p) =(\frac{m}{p})  p^{\frac{1}{2}}  G_m(\ell)$). The contribution of these terms to \eqref{expression-in-G} is 
$$
G_m (\ell)   \left(\frac{ \pm\kappa  v\beta N_0}{ \ell}\right)
\sum_{(p ,\alpha \ell m v N_0)=1}\frac{\Lambda_\mathcal{L}(p)}{p} \left(\frac{ \pm\kappa m v\beta N_0}{ p}\right)e\bigg( \frac{ma \overline{ v\beta\alpha^2 \ell  p}}{N_0} \bigg)
H \left(\frac{\log p}{L} \right)
\widehat{\Phi}\left( \frac{Xm }{\ell  p N_0 v\beta\alpha^2}\right)
 .
$$
Also, as $\widehat{\Phi}(s)$ decays rapidly, we may further consider 
\begin{equation}\label{p-range}
p > \frac{X^{1- \varepsilon} |m|}{\ell   N_0 v\beta\alpha^2 }.
\end{equation}
(Otherwise, we have $\frac{X |m|}{\ell  p N_0 v\beta\alpha^2 } \ge X^{ \varepsilon} $. This forces $\widehat{\Phi}\left( \frac{Xm }{\ell  p N_0 v\beta\alpha^2}\right) \ll_K X^{-K\varepsilon} $ for any $K>0$, which gives a negligible error.)
By splitting $p$ into arithmetic progressions modulo $N_0$, it suffices to estimate
$$
G_m (\ell) \sum_{\substack{ (p ,\alpha \ell m v N_0)=1 \\ p\equiv c \mymod{N_0} }}\frac{\Lambda_\mathcal{L}(p)}{p} \left(\frac{ p}{q}\right)
H\left(\frac{\log p}{L} \right) 
\widehat{\Phi}\left( \frac{Xm }{\ell p   N_0 v\beta\alpha^2}\right),
$$
for $c$ coprime to $N_0$ and  $q\mid m v\beta$, where the sum runs over $p$ satisfying \eqref{p-range}.   Under GRH for $L(s,E)$ and its twists by quadratic characters and Dirichlet characters modulo $N_0$, Abel's summation then yields the  bound
$$
\ll M|G_m (\ell)| \left(\frac{\ell   N_0 v\beta\alpha^2 }{X^{1-\varepsilon} |m|} \right)^{\frac{1}{2}} \log^2 X
\ll  M \frac{ X^{\varepsilon} \ell^{\frac{3}{2}}   (v\beta)^{\frac{1}{2}} \alpha }{\sqrt{ X |m|}}.
$$

For the case that $n$ is a power of a prime dividing $m$. As $G_m(\ell n)=0$ for $n\nmid m^2$, we may assume  $n\mid m^2$. Again, by the rapid decay of $\widehat{\Phi}(s)$, we may assume
$$
n > \frac{X^{1- \varepsilon} |m|}{\ell N_0 v\beta\alpha^2 }.
$$
Since 
$
|G_m(\ell n  )| \le (|m| \ell n)^{\frac{1}{2}},
$
the contribution of these terms to \eqref{expression-in-G} is
$$
\ll \sum_{n\mid m^2} M \Lambda(n) \frac{(|m| \ell )^{\frac{1}{2}}}{ (X^{1- \varepsilon} |m|)/(\ell N_0 v\beta\alpha^2 )}
\ll M(\log m) X^{\varepsilon} \frac{\ell^{\frac{3}{2}} v\beta\alpha^2}{ X |m|^{\frac{1}{2}}}
\ll  \frac{M X^{2\varepsilon} \ell^{\frac{3}{2}}   (v\beta)^{\frac{1}{2}} \alpha }{\sqrt{ X |m|}},
$$
as $\log m \ll \log X$, provided that $\alpha\le A\le \sqrt{X}$.

Thus, \eqref{rearr} becomes
$$
\ll \frac{X}{\ell L N_0} \sum_{ 0<|m| \le \ell e^L A^2 X^{-1+ 3\varepsilon} } \sum_{\beta\mid v} \sum_{  \alpha\le A }   \frac{1}{  v\beta\alpha^2} \cdot  \frac{  M X^{2\varepsilon} \ell^{\frac{3}{2}}   (v\beta)^{\frac{1}{2}} \alpha }{\sqrt{ X |m|}}
\ll X^{\frac{1}{2} +3\varepsilon}  \ell^{\frac{1}{2}}   \sum_{ 0<|m| \le \ell e^L A^2 X^{-1+ 3\varepsilon} } \frac{\log A}{\sqrt{  |m|}},
$$
which is 
$$
\ll M \ell  e^{\frac{L}{2}} A X^{5\varepsilon},
$$
provided that $ A\le \sqrt{X}$.
Balancing this with \eqref{1st-error}, we shall choose
\begin{equation}\label{choiceA}
A = \X^{\frac{1}{2} -2\varepsilon}  \ell^{-\frac{1}{2} } e^{-\frac{L}{4}},
\end{equation}
which gives the both error is at most $ O(X^{\frac{1}{2}+ 3\varepsilon}\ell^{\frac{1}{2}}  e^{\frac{L}{4}}).$

Finally, we focus on the main term of $S_2$ arising from terms with $m=0$. Recall that $\tau_0(\ell n) = \phi(\ell n)$ if $\ell n$ is a square, and $\tau_0(\ell n) = 0$ otherwise. From  \eqref{alpha<A}  and \eqref{inner-expression}, combined with the above discussion for errors, it follows that  the main term of $S_2$ is
 \begin{align}\label{mid-main-S2}
\begin{split}
 \frac{X}{v L N_0} \widehat{\Phi}(0) \sum_{\beta\mid v}  \frac{\mu (\beta)}{  \beta}\sum_{\substack{ ( n , vN_0)=1\\ \ell n=\square}}\sum_{ \substack{(\alpha ,\ell n  vN_0)=1 \\ \alpha\le A }}   \frac{\mu (\alpha)}{  \alpha^2}
  \frac{\Lambda_\mathcal{L}(n)}{\sqrt{n}} \frac{\phi(\ell n)}{\ell n} 
\left(\widehat{h} \left(\frac{\log n}{L} \right) +\widehat{h} \left(-\frac{\log n}{L} \right) \right),
  \end{split}
\end{align}
where we used the condition $(\ell , v)=1$. 

Observe that as $n$ in the sum of $S_2$ must be a prime power, to have $n\ell=\square$, $\ell$ can only be a square, or a prime times a square. Firstly, when $\ell$ is a square, writing $n=r^2$, by a direct calculation, we can express the inner triple sum in \eqref{mid-main-S2} as
$$
\sum_{(r,vN_0)=1 } \frac{\Lambda_\mathcal{L}(r^2)}{r}
\bigg(  \prod_{p\mid r \ell v} \left(1 + \frac{1}{p} \right)^{-1}
 \prod_{p\nmid N_0 } \left(1 - \frac{1}{p^2} \right)  +O\left( \frac{v^\varepsilon}{A}\right) \bigg)  
\left(\widehat{h} \left(\frac{\log r^2}{L} \right) +\widehat{h} \left(-\frac{\log r^2}{L} \right) \right)  .
$$
Recalling that the theory of Rankin-Selberg $L$-functions gives
$$
\sum_{ \substack{\p\le y\\ (p,N_0)=1}} \frac{\Lambda_{\mathcal{L}}(p^2)}{p} 
= \sum_{j=1}^M\sum_{ \substack{\p\le y\\ (p,N_0)=1}} \frac{\Lambda_{f_j}(p^2)}{p} =M( -\log y +O(1)),
$$
and applying Abel's summation, we then establish that \eqref{mid-main-S2} equals
\begin{align*}
&-  \frac{X}{vLN_0} \widehat{\Phi}(0)
 \bigg(  \prod_{p\mid  \ell v} \left(1 + \frac{1}{p} \right)^{-1}
 \prod_{p\nmid N_0 } \left(1 - \frac{1}{p^2} \right)  +O\left( \frac{v^\varepsilon}{A}\right) \bigg)\\
&\times    \left(
 \int_1^{\infty}  \left(\widehat{h} \left(\frac{2\log y}{L} \right) +\widehat{h} \left(-\frac{2\log y}{L} \right) \right) \frac{dy}{y} +O(1) \right)\\
 &= -  \frac{MX}{vN_0} \widehat{\Phi}(0) \prod_{p\mid  \ell v} \left(1 + \frac{1}{p} \right)^{-1} 
 \prod_{p\nmid N_0 } \left(1 - \frac{1}{p^2} \right)
 \frac{h(0)}{2} +O\left(\frac{MX v^\varepsilon}{vA} + \frac{MX}{vL}\right),
\end{align*}
(cf. \cite[Eq. (30)]{RaSo}), where the error term involving $A$ is acceptable by our choice of $A$ in \eqref{choiceA}.

Lastly, if $\ell$ is $q$ times a square, with a (unique) prime $q$, we must have $n=q r^2$ for some $r$. However, as $n$ has to be a prime power in $S_2$, it can only be an odd power of $q$. Therefore, we can bound \eqref{mid-main-S2}, the main term of $S_2$, by
$$
\ll \frac{MX}{vLN_0} \frac{\log q}{\sqrt{q}}
  \prod_{p\mid  \ell v} \left(1 + \frac{1}{p} \right)^{-1}
 \prod_{p\nmid N_0 } \left(1 - \frac{1}{p^2} \right) 
$$
(cf. \cite[Eq. (31)]{RaSo}), which concludes the proof.

\section{Moment calculations}

In this section, we shall prove Propositions \ref{M-forms-moments} and \ref{P-C-moments}.



\subsection{Proof of Proposition \ref{M-forms-moments}}\label{proof-2.2}
It follows from \eqref{key-prop-RS-original} with $v=1$ that for $n\in\Bbb{N}$ coprime to $N_0$ such that $\sqrt{n}\le X^{\frac{1}{2}-\varepsilon}$, we have
$$
 \sum_{d\in\mathcal{F}(\kappa,a)} \chi_d(n) \Phi\left(\frac{\kappa d}{X}\right)
= \delta(n=\square)
\frac{X}{N_0} \prod_{p\mid n } \left(1 + \frac{1}{p} \right)^{-1}
 \prod_{p\nmid N_0 } \left(1 - \frac{1}{p^2} \right) \widehat{\Phi}(0)
 + O(X^{\frac{1}{2}+\varepsilon}n^{\frac{1}{2}}),
$$
where $\delta(n=\square)=1$ if $n$ is a square, and $\delta(n=\square)=0$ otherwise. Hence, expanding out  $\mathcal{P}_{\mathbf{a}}(x;d)^k$, we can express the left of \eqref{M-P-moments} as
\begin{align*}
 \sum_{\substack{p_1,\ldots p_k\le x \\  p_i\nmid N_0}} 
 \frac{ \prod_{i=1}^k (a_{1}\lambda_{f_1}(p_i)+ \cdots +a_{M}\lambda_{f_M} (p_i)){ w(p_i) }  }{\sqrt{p_1 \cdots p_k}} \sum_{d\in\mathcal{F}(\kappa,a)}  \chi_d(p_1 \cdots p_k) \Phi\left(\frac{\kappa d}{X}\right).
\end{align*}
Observe that if $p_1\cdots p_k$  is not a square (which is always the case if $k$ is odd), then the inner sum over $d$ is $O(X^{\frac{1}{2}+\varepsilon}(p_1 \cdots p_k)^{\frac{1}{2}})$. Thus, Deligne's bound $|\lambda_{f_j} (p_i)|\le 2$ implies that the contribution of each  $p_1,\ldots, p_k$ such that  $p_1\cdots p_k$  is not a square is $O_{\mathbf{a}, k }(X^{\frac{1}{2} +\varepsilon}) $. Consequently, when $k$ is odd, the left of \eqref{M-P-moments} is $O_{\mathbf{a}, k }( x^k X^{\frac{1}{2} +\varepsilon})$.

Hence, it remains to consider the case of even $k$ and estimate
$$
\frac{X}{N_0}
 \prod_{p\nmid N_0 } \left(1 - \frac{1}{p^2} \right) \widehat{\Phi}(0) 
  \sum_{\substack{p_1,\ldots, p_k\le x \\  p_i\nmid N_0 \\ p_1\cdots p_k =\square}} 
 \frac{ \prod_{i=1}^k (a_{1}\lambda_{f_1}(p_i)+ \cdots +a_{M}\lambda_{f_M} (p_i)) { w(p_i) }}{\sqrt{p_1 \cdots p_k}} \prod_{p\mid p_1\cdots p_k } \left(1 + \frac{1}{p} \right)^{-1}.
$$
Denote the distinct primes appearing in  $p_1\cdots p_k $ by $q_1< \cdots < q_r$ with  multiplicity $b_1,\ldots, b_r$, respectively. As each $b_j\ge 2$ and $\sum_{j=1}^r b_j =r$, we deduce $r\le \frac{k}{2}$. Clearly, the contribution of $q_1< \cdots < q_r$ with $r< \frac{k}{2}$ is
$$
\ll X \bigg( \sum_{\substack{p \le x \\  p \nmid N_0 }} 
 \frac{ (a_{1}\lambda_{f_1}(p)+ \cdots +a_{M}\lambda_{f_M} (p))^2}{p} \bigg)^{\frac{k}{2} -1}
\ll_{\mathbf{a},r} X \bigg( \sum_{p \le x } \frac{1}{p}
\bigg)^{\frac{k}{2} -1}   \ll  X(\log\log x)^{\frac{k}{2} -1},
$$
which is negligible as $r\le  \frac{k}{2} -1 $. 
When $r= \frac{k}{2}$, each $b_j$ must be equal to 2, and so the contribution of $q_1< \cdots < q_r$ with $r= \frac{k}{2}$ is
$$
\frac{X}{N_0}
 \prod_{p\nmid N_0 } \left(1 - \frac{1}{p^2} \right) \widehat{\Phi}(0) 
\frac{k!}{2^{k/2} (k/2)!} 
  \sum_{\substack{q_1,\ldots, q_{k/2}\le x \\ q_\ell \text{ distinct} \\  q_\ell\nmid N_0}} 
 \frac{ \prod_{\ell=1}^{k/2} ( a_{1}\lambda_{f_1}(q_\ell)+ \cdots +a_{M}\lambda_{f_M} (q_\ell) )^2 { w(q_\ell)^2 } }{(q_1+1) \cdots (q_{k/2}+1) } .
$$
(For the rationale of the appearance of the factor $\frac{k!}{2^{k/2} (k/2)!}$, see \cite[p. 5, especially Eq. (4.2)]{Ha}.) Moreover, from the estimate
$$
\sum_{p \le x } 
 \frac{ \lambda_{f_j}(p) \lambda_{f_{j'}} (p)}{p}  =\delta_{jj' } \log\log x +O(1), 
$$
where $\delta_{jj' } =1$ if $j=j'$, and $\delta_{jj' }=0$ otherwise, it then follows that the  inner sum over $q_\ell$ is
$$
 X \bigg( \sum_{\substack{p \le x \\  p \nmid N_0 }} 
 \frac{ (a_{1}\lambda_{f_1}(p)+ \cdots +a_{M}\lambda_{f_M} (p))^2}{p} \bigg)^{\frac{k}{2}}
\ll X \bigg( \sum_{j=1}^M  a_j^2 \log\log x +O_{\mathbf{a}}(1) 
\bigg)^{\frac{k}{2}},
$$
which, combined with the above discussion, yields the first assertion.

To prove \eqref{M-P-h-moments}, we expand $\mathcal{P}_{\mathbf{a}}(x;d)$ as above so that the left of \eqref{M-P-h-moments} becomes
$$
 \sum_{\substack{p_1,\ldots, p_k\le x \\  p_i\nmid N_0}} 
 \frac{ \prod_{i=1}^k (a_{1}\lambda_{f_1}(p_i)+ \cdots +a_{M}\lambda_{f_M} (p_i))w(p_i)}{\sqrt{p_1 \cdots p_k}} \sum_{d\in\mathcal{F}(\kappa,a)}
  \sum_{j=1}^M \sum_{\gamma_{j,d}} h\left( \frac{\gamma_{j,d} L}{2\pi}\right) \chi_d(p_1 \cdots p_k) \Phi\left(\frac{\kappa d}{X}\right).
$$
By the first part of Proposition 2, we know that the contribution of the terms arising from $p_1\cdots p_k$ that is not a square nor a prime times a square contribute is
$$
\ll\sum_{\substack{p_1,\ldots p_k\le x \\  p_i\nmid N_0}} 
 \frac{ \prod_{i=1}^k (a_{1}\lambda_{f_1}(p_i)+ \cdots +a_{M}\lambda_{f_M} (p_i))}{\sqrt{p_1 \cdots p_k}} X^{\frac{1}{2} +\varepsilon} (p_1 \cdots p_k)^{\frac{1}{2}} e^{\frac{L}{4}}
 \ll_{\mathbf{a},k} x^k  X^{\frac{1}{2} +\varepsilon} e^{\frac{L}{4}}.
$$

If $p_1\cdots p_k$  is a square (which happens only if $k$ is even), by the second part of Proposition 2,  the main term is
\begin{align*}
& \frac{  M X}{N_0}
 \prod_{p\nmid N_0 } \left(1 - \frac{1}{p^2} \right) \widehat{\Phi}(0)
 \left( \frac{2\log X}{L} \widehat{h}(0)  +\frac{h(0)}{2}+ O(L^{-1}) \right)\\
& \times
  \sum_{\substack{p_1,\ldots p_k\le x \\  p_i\nmid N_0 \\ p_1\cdots p_k =\square}} 
 \frac{ \prod_{i=1}^k (a_{1}\lambda_{f_1}(p_i)+ \cdots +a_{M}\lambda_{f_M} (p_i))w(p_i)}{\sqrt{p_1 \cdots p_k}} \prod_{p\mid p_1\cdots p_k } \left(1 + \frac{1}{p} \right)^{-1}.
\end{align*}
As the inner sum can be estimated in the exact same way as above, we then conclude that the main term  is
\begin{align*}
 \frac{ M  X}{N_0}
 \prod_{p\nmid N_0 } \left(1 - \frac{1}{p^2} \right) \widehat{\Phi}(0)
 \left( \frac{2\log X}{L} \widehat{h}(0)  +\frac{h(0)}{2}+ O(L^{-1}) \right)
(M_k + o(1))(\log \log X)^{\frac{k}{2}}.
\end{align*}

Finally, for odd $k$, by the third part of Proposition 2, we can bound the contribution of the
terms arising from $p_1\cdots p_k$ that is a prime times a square by
\begin{align*}
&\ll
\frac{X}{LN_0} \bigg|  \sum_{ \substack{q\le x \\ q\nmid N_0}} \frac{(\log q)((a_{1}\lambda_{f_1}(q)+ \cdots +a_{M}\lambda_{f_M} (q))}{ q}  \bigg|\\
&\times  \sum_{\substack{p_1,\ldots p_{k-1}\le x \\  p_i\nmid N_0 \\ p_1\cdots p_{k-1} =\square}} 
 \frac{ \prod_{i=1}^{k-1} (a_{1}\lambda_{f_1}(p_i)+ \cdots +a_{M}\lambda_{f_M} (p_i))}{\sqrt{p_1 \cdots p_{k-1}}} \prod_{p\mid p_1\cdots p_{k-1} } \left(1 + \frac{1}{p} \right)^{-1}
 +x^k X^{\frac{1}{2}+\varepsilon} e^{\frac{L}{4}}\\
 &\ll \frac{X}{ LN_0} (\log\log X)^{\frac{ k-1}{2}}  +x^k X^{\frac{1}{2}+\varepsilon} e^{\frac{L}{4}},
\end{align*}
as desired, where we used the GRH estimate 
$
\sum_{q\le t } (\log q)\lambda_{f_j}(q) 
\ll t^{\frac{1}{2}} \log^2 (N_0t)
$
to deduced that the first sum (over $q$) is $O(1)$.

\subsection{Proof of Proposition \ref{P-C-moments}}\label{pf-PC}

As the first part of the proposition follows from the virtue of the proof of \cite[Proposition 7]{RaSo15},\footnote{Indeed, the presence of $b$ and $c$ would add the factor $c^j (b+c)^{k-j}$ to the sum in \cite[Eq. (15)]{RaSo15}, and it could be easily checked that such an extra factor does not affect the proof but appears in the first displayed asymptotics in \cite[p. 1051]{RaSo15}. This leads to the change of the corresponding terms of $\log\log X$ from $( (\sigma(E)^2 -1) \log\log X)^{j/2} ( \log\log X)^{(k-j)/2}$ to 
$(c^2 (\sigma(E)^2 -1) \log\log X)^{j/2} ( (b+c)^2 \log\log X)^{(k-j)/2}$ (note that both $j$ and $k-j$ are even here). From which (and applying the binomial theorem), one can derive the claimed generalisation \eqref{P-C-moments-1}.}  we shall omit it and prove \eqref{P-C-h-moments}.
%
%
%
%
%
To begin, opening the sum, we see that the left of \eqref{P-C-h-moments} is
\begin{equation}\label{left=}
\sum_{j=0}^{k}  {k \choose j}  (-c)^j (b+c)^{k-j} \sum_{d\in\mathcal{F}(\kappa,a)} 
\sum_{\log X\le p_1,\ldots, p_j\le x}C_{p_1}(d)\cdots C_{p_j}(d)  \mathcal{P}(d; x)^{k-j} 
  \sum_{\gamma_{d}} h\left( \frac{\gamma_{d} L}{2\pi}\right)\Phi\left(\frac{\kappa d}{X}\right).
\end{equation}
As noted in \cite[pp. 1048-1049]{RaSo15}, if $q_1 < q_2 < \cdots < q_u$ are the distinct primes appearing in $p_1,\ldots, p_j$,  denoting the multiplicity of $q_i$  by $m_i$, one has
$$
C_{p_1}(d)\cdots C_{p_j}(d)  =\prod_{i=1}^u  C_{q_i}(d)^{m_i}
= \sum_{v\mid (d,q_1\cdots q_u)} 
\sum_{rs=v}\mu(r)
\prod_{i=1}^u  C_{q_i}(s)^{m_i}.
$$
Hence, the inner sum (over $d$) of \eqref{left=} can be written as
\begin{equation}\label{left=-2}
\sum_{\log X\le p_1,\ldots, p_j\le x}\sum_{v\mid q_1\cdots q_u} 
\sum_{rs=v}\mu(r)
\prod_{i=1}^u  C_{q_i}(s)^{m_i}\sum_{\substack{ d\in\mathcal{F}(\kappa,a)\\ v\mid d}} 
 \mathcal{P}(d; x)^{k-j} 
  \sum_{\gamma_{d}} h\left( \frac{\gamma_{d} L}{2\pi}\right)\Phi\left(\frac{\kappa d}{X}\right).
\end{equation}
(Note that  $v\le x^k= X^{k/ \log \log \log X}$ for $v$ appearing in \eqref{left=-2}; we shall use this repeatedly in the remaining discussion.)
Expanding out $\mathcal{P}(d; x)^{k-j} $, we can express the last double sum in \eqref{left=-2} as
\begin{equation}\label{main-double}
 \sum_{\substack{p_{j+1},\ldots, p_k\le x \\  p_i\nmid vN_0}} 
 \frac{ \lambda_{E}(p_{j+1})\cdots\lambda_{E} (p_k) w(p_{j+1})\cdots w (p_k)}{\sqrt{p_{j+1} \cdots p_k}} \sum_{\substack{ d\in\mathcal{F}(\kappa,a)\\ v\mid d}}  
 \sum_{\gamma_{d}} h\left( \frac{\gamma_{d} L}{2\pi}\right) \chi_d(p_{j+1} \cdots p_k) \Phi\left(\frac{\kappa d}{X}\right).
\end{equation}
(Note that if $p_i \mid v$ for some $i$, then $\chi_d(p_{j+1} \cdots p_k)=0$ as $v\mid d$.) Applying the first part of Proposition \ref{key-prop-2}, we deduce that the contribution of the terms, arising from $p_{j+1}\cdots p_k$ that is not a square nor a prime times a square, to \eqref{main-double} is
\begin{equation}\label{main-double-error}
\ll\sum_{p_{j+1},\ldots, p_k\le x } 
\frac{| \lambda_{E}(p_{j+1})\cdots\lambda_{E} (p_k)|}{\sqrt{p_{j+1} \cdots p_k}} 
 X^{\frac{1}{2} +\varepsilon} (p_{j+1} \cdots p_k)^{\frac{1}{2}} e^{\frac{L}{4}}
 \ll_{k} x^k  X^{\frac{1}{2} +\varepsilon} e^{\frac{L}{4}}
\end{equation}
as $|\lambda_E(p)|\le 2$.

If $p_{j+1}\cdots p_k$  is a square (when this happens, $k-j$ must be even), by the second part of Proposition \ref{key-prop-2}, the main term of \eqref{main-double} is
\begin{align}\label{P-C-main}
 \begin{split}
& \frac{X}{vN_0}
 \prod_{p\nmid N_0 } \left(1 - \frac{1}{p^2} \right) 
 \prod_{p\mid v} \left( 1+\frac{1}{p}\right)^{-1}
 \widehat{\Phi}(0)
 \left( \frac{2\log X}{L} \widehat{h}(0)  +\frac{h(0)}{2}+ O(L^{-1}) \right)\\
& \times
  \sum_{\substack{p_{j+1},\ldots, p_k\le x \\  p_i\nmid vN_0 \\ p_{j+1}\cdots p_k =\square}} 
 \frac{ \lambda_{E}(p_{j+1})\cdots\lambda_{E} (p_k)  w(p_{j+1})\cdots w (p_k)}{\sqrt{p_{j+1} \cdots p_k} }  \prod_{p\mid p_{j+1}\cdots p_k } \left(1 + \frac{1}{p} \right)^{-1}
  \end{split}
\end{align}
(while the error is still $\ll_{k} x^k  X^{\frac{1}{2} +\varepsilon} e^{\frac{L}{4}}$ as argued above). Observe that the terms above with some $p_i$ dividing $v$ contribute at most
$$
\ll \frac{ X(\log X)(\log \log X)^{k}}{ v L\log X}
$$
as every prime factor of $v$ is larger than $\log X$. Therefore, removing the condition $ p_i\nmid v$ for $j +1 \le i \le k$ results in an error $\ll X(\log X)(\log \log X)^{k}/(vL \log X)$. Now, we further denote the distinct primes appearing in  $p_{j+1}\cdots p_k $ by $q'_1, \ldots, q'_{u'}$ with  multiplicity $m'_1,\ldots, m'_{u'}$, respectively. Note that as each $m'_i\ge 2$ and $\sum_{i=1}^{u'} m'_i = k-j$, we know $u\le \frac{k-j}{2}$. The contribution of $q'_1, \ldots, q'_{u'}$ with $u'< \frac{k-j}{2}$ (which forces $u'\le  \frac{k-j}{2} -1 $) to \eqref{P-C-main} is
$$
\ll \frac{X}{v}  \prod_{p\mid v}\left( 1+\frac{1}{p}\right)^{-1} \frac{\log X}{L} 
 \bigg( \sum_{\substack{p \le x \\  p \nmid N_0 }} 
 \frac{ \lambda_{E}(p)^2 w(p)^2 }{p} \bigg)^{\frac{k-j}{2} -1}
\ll  \frac{X\log X}{vL} (\log\log X)^{\frac{k-j}{2} -1},
$$
which is negligible. Lastly, for $u'= \frac{k-j}{2}$, we know that each $m_i$ must equal 2, and thus the contribution of $q'_1, \ldots, q'_{(k-j)/2}$ (to the sum over $p_i$ in \eqref{P-C-main}) is
$$
\frac{(k-j)!}{2^{(k-j)/2} ((k-j)/2)!} 
  \sum_{\substack{q'_1,\ldots, q'_{(k-j)/2}\le x \\ q'_i \nmid N_0}} 
 \frac{ \prod_{i=1}^{(k-j)/2} \lambda_{E}(q'_i)^2}{(q'_1+1) \cdots (q'_{(k-j)/2}+1) } .
$$
(A justification of the factor $\frac{(k-j)!}{2^{(k-j)/2} ((k-j)/2)!}$ can be found in \cite[p. 5]{Ha}.)
Observe that
$$
 \sum_{\substack{q \le x \\  q \nmid N_0 }} 
 \frac{ \lambda_{E}(q)^2 w(q)^2}{q +1} 
 =  \sum_{\substack{q \le x \\  q \nmid N_0 }} 
 \frac{ \lambda_{E}(q)^2 w(q)^2}{q} + O(1)
 = \sum_{\substack{q \le x \\  q \nmid N_0 }} 
 \frac{ \lambda_{E}(q)^2   }{q}  +  O(1)
 = \log\log x + O(1),
$$
where the middle equality follows from the mean value theorem and the elementary estimate $\sum_{p\le y} \frac{\log p}{p} \ll \log y$. Hence, we then conclude that the main term of \eqref{main-double} is equal to
\begin{align}\label{P-C-main-middle}
\begin{split}
 &\frac{X}{vN_0}
 \prod_{p\nmid N_0 } \left(1 - \frac{1}{p^2} \right)\prod_{p\mid v}\left( 1+\frac{1}{p}\right)^{-1}  \widehat{\Phi}(0)
 \left( \frac{2\log X}{L} \widehat{h}(0)  +\frac{h(0)}{2}+ O(L^{-1}) \right)\\
&\times \left(\frac{(k-j)!}{2^{(k-j)/2} ((k-j)/2)!} + o(1)\right)(\log \log X)^{\frac{k-j}{2}}
 \end{split}
\end{align}
in this case (namely, the contribution of terms with $p_{j+1}\cdots p_k =\square$).

Finally, by the third part of Proposition \ref{key-prop-2}, we can bound the contribution of the
terms arising from $p_{j+1}\cdots p_k$ that is a prime times a square (which forces $k-j$ to be odd) by
\begin{align}\label{P-C-main-middle'}
\begin{split}
&\ll
\frac{X}{vLN_0}   \sum_{ q\le x } \frac{\log q}{ q}  
 \prod_{p\mid v} \left( 1+\frac{1}{p}\right)^{-1}
 \sum_{\substack{p_{j+1},\ldots, p_{k-1}\le x \\  p_i\nmid N_0 \\ p_{j+1}\cdots p_{k-1} =\square}} 
 \frac{ \prod_{i=j+1}^{k-1} \lambda_{E}(p_i)}{\sqrt{p_{j+1} \cdots p_{k-1}}} 
 +x^k X^{\frac{1}{2}+\varepsilon} e^{\frac{L}{4}}\\
 &\ll \frac{X \log x}{v LN_0} \prod_{p\mid v} \left( 1+\frac{1}{p}\right)^{-1} (\log\log X)^{\frac{ k-j-1}{2}}  +x^{k} X^{\frac{1}{2}+\varepsilon} e^{\frac{L}{4}}.
 \end{split}
\end{align}

In light of the above discussion (especially, \eqref{main-double}, \eqref{main-double-error}, \eqref{P-C-main-middle}, and \eqref{P-C-main-middle'}), to handle  \eqref{left=-2}, we shall estimate
\begin{equation}\label{Cq-sum}
\sum_{\log X\le p_1,\ldots, p_j\le x}\sum_{v\mid q_1\cdots q_u} 
\sum_{rs=v}\mu(r)
\prod_{i=1}^u  C_{q_i}(s)^{m_i} \frac{1}{v}
 \prod_{p\mid v} \left( 1+\frac{1}{p}\right)^{-1}
\end{equation}
(for $j\ge 1$).
Noting that $\frac{1}{v}
 \prod_{p\mid v}  ( 1+\frac{1}{p} )^{-1}
 =  \prod_{p\mid v}  \frac{1}{p+1} $ (as  $v$ is square-free  in our consideration), 
 we follow \cite[p. 1050]{RaSo15} to set
$$
G(q_1^{m_1} \cdots  q_u^{m_u})
=\sum_{v\mid q_1\cdots q_u} 
\sum_{rs=v}\mu(r)
\prod_{i=1}^u  C_{q_i}(s)^{m_i}
\prod_{p\mid v}\frac{1}{p+1},
$$
which is multiplicative. In addition, we recall that
$$
G(p^\alpha) = ( \log c(p)  )^\alpha \left( \frac{1}{p+1}  \left(1- \frac{1}{p+1} \right)^\alpha
+\frac{p}{p+1} \left( \frac{-1}{p+1} \right)^\alpha \right),
$$
which implies that $G(q_1^{m_1} \cdots  q_u^{m_u})$ is non-vanishing only if  $m_i\ge 2$ for all $i$, and also 
$$
G(p^\alpha) \ll (\log c(p))^\alpha /p
$$
for all $\alpha\ge 2$.

Given $q_1 < \cdots < q_u$ and $m_i \ge 2$ with $\sum_{i=1}^{u} m_i = j$, the number of
choices for $p_1, \ldots, p_j$ 
is $
\frac{j!}{m_1! \cdots m_u!}.$ Hence, the first double sum in \eqref{Cq-sum} can be written as
\begin{equation}\label{G-sum}
\sum_{\substack{u\\  m_i\ge 2 \,\forall i\\ \sum_{i=1}^{u} m_i = j }}
\frac{j!}{m_1! \cdots m_u!} \sum_{\log X \le q_1<\cdots<q_u< x}
G(q_1^{m_1} \cdots  q_u^{m_u}).
\end{equation}
If  $m_i \ge 3$ for some $i$, then the above inner sum over $q_i$  is at most
\begin{equation}\label{G-sum-error}
\ll(\log \log X)^{(j-1)/2}.
\end{equation}
Otherwise, if $m_i = 2$ for all $i$, we must have $j = 2 u$ (which is even). As a direct calculation shows
$$ 
  G(p^2)=  \frac{ (\log c(p))^2}{p}  + O\left(\frac{1}{p^2}\right), 
$$  
it then follows from \eqref{est-cp2} that the terms with all $m_i=2$ contribute
\begin{equation}\label{G-sum-main}
 \frac{j!}{2^{j/2} (j/2)!}( (\sigma(E)^2 -1) \log\log X +O( \log\log\log X))^{j/2} 
\end{equation}
to \eqref{G-sum}. Therefore, when $k$ is even, it follows from \eqref{left=-2}, \eqref{main-double}, \eqref{main-double-error}, \eqref{P-C-main-middle},  \eqref{Cq-sum},  and  \eqref{G-sum-main} that 
\eqref{left=} (and thus the left of \eqref{P-C-h-moments}) becomes
\begin{align*}
& \sum_{\substack{ j=0\\ \text{$j$ even}}}^{k}  {k \choose j}  (-1)^j   \frac{X}{N_0}
 \prod_{p\nmid N_0 } \left(1 - \frac{1}{p^2} \right)  \widehat{\Phi}(0)
 \left( \frac{2\log X}{L} \widehat{h}(0)  +\frac{h(0)}{2}+ O(L^{-1}) \right)\\
&\times \frac{j!}{2^{j/2} (j/2)!}(c^2 (\sigma(E)^2 -1 +o(1) ) \log\log X)^{j/2} 
  \frac{(k-j)!}{2^{(k-j)/2} ((k-j)/2)!}( ((b+c)^2  +o(1)) \log \log X)^{\frac{k-j}{2}},
 \end{align*}
which yields the claimed estimate for even $k$. Here, we use the consequence of \eqref{P-C-main-middle'} and \eqref{G-sum-error} that, in general, the contribution of those terms, with either $j$ or $k-j$ being odd, to \eqref{left=} is at most
\begin{align*}
\ll_{b,c,k}    \frac{X\log X }{ LN_0}(\log\log X)^{\frac{ k-1}{2}}.
\end{align*}
Finally, we complete the proof by noting that this estimate also implies that for odd $k$, the ``main term'' of \eqref{left=} does not manifest. Indeed, it is  $\frac{X\log X }{ LN_0} o((\log\log X)^{\frac{ k}{2}})$ as required by \eqref{P-C-h-moments}.

\section{Remarks on the conjecture of {Radziwi\l\l} and Soundararajan}\label{rmkRS}

Let $E$ be an elliptic curve over $\Bbb{Q}$, and let $\mathcal{E}$ be defined as in \eqref{def-E}.  Note that it follows from \cite[Proposition 1]{RaSo} that for every $d\in\mathcal{E}$ with $X\le |d|\le 2X$, if  $L(\frac{1}{2},E_{d})\neq 0$, then under GRH, one has
\begin{equation}\label{RaSo-prop1}
 \log L(\oh, E_d) = \mathcal{P}(d; x) -\oh \log \log X 
+O\bigg( \log\log\log X +
  \sum_{\gamma_{d}} \log \bigg(1+ \frac{1}{(\gamma_{d} \log x)^2} \bigg)  \bigg),
\end{equation}
where $\oh+ i\gamma_{d}$ ranges over the non-trivial zeros of $ L(s,E_{d})$. For $d\in\mathcal{E}$ with $L(\frac{1}{2},E_d)\neq 0$, defining
$$
S_0(E_d) = L(\oh, E_d)  \frac{|E_d(\Bbb{Q})_{\mathrm{tors}}|^2}{\Omega(E_d)\Tam(E_d)},
$$
by \eqref{RaSo-prop1} and the argument leading to \eqref{key-expression-SEd}, we have
\begin{align*}
 \begin{split}
&\log (S_0(E_d)/\sqrt{|d|} ) - \mu(E) \log\log |d|\\
& =   \mathcal{P}(d; x) - \mathcal{C}(d;x)
 +O\bigg( (\log\log\log X)^2 +
  \sum_{\gamma_{d}} \log \bigg(1+ \frac{1}{(\gamma_{d} \log x)^2} \bigg)  \bigg)
  \end{split}  
\end{align*}
for all but at most $\ll X/\log\log\log X$  $d\in\mathcal{E}$ with $X\le |d|\le 2X$. As both Propositions \ref{key-prop-2} and \ref{P-C-moments} are also valid with $\mathcal{F} (\kappa, a)$ being replacing by
$$
\mathcal{E} (\kappa, a) = \{d \in \mathcal{E} : \kappa d > 0,\enspace d \equiv a\enspace \mymod{N_0}\},
$$
processing a similar argument as in Section \ref{proof-of-main-thm} (together with two approximations above) then yields the following variant of Theorem \ref{main-thm}.

\begin{thm}
Assume GRH for the family of twisted $L$-functions $L(s,E \otimes \chi)$ with all Dirichlet characters $\chi$.
For any fixed $ \underline{\alpha}=(\alpha_1,\alpha_2)$ and  $\underline{\beta}=(\beta_1,\beta_2)$,  as $ X \rightarrow \infty$, 
\begin{align*}
\begin{split}
\#\bigg\{ d\in\mathcal{E},
X< |d|\le 2X :\, & \frac{ \log L(\oh,E_{d}) + \frac{1}{2}\log \log |d| }{\sqrt{ \log\log |d|}}\in (\alpha_1,\beta_1),\\
& \frac{\log (S_0(E_d)/\sqrt{|d|} ) - \mu(E) \log\log |d|}{\sqrt{\sigma(E)^2 \log\log |d|}}
\in (\alpha_2,\beta_2)\bigg\} 
\end{split}
\end{align*}
is greater or equal to
$$
\frac{1}{4} (\Xi_E(\underline{\alpha},\underline{\beta})+o(1) )
\# \{ d \in\mathcal{E} :  X< |d|\le 2X \}, 
$$
where $\Xi_E$ is defined as in \eqref{def-XiE-KE}. Moreover, under BSD for elliptic curves with analytic rank zero, the above assertion is true with $S_0(E_d)$ being replaced by $|\Sha(E_d)|$.
\end{thm}

In closing this section,  with \cite[Conjecture 1]{RaSo} of {Radziwi\l\l} and Soundararajan, Conjecture \ref{R1-conj}, and the result above in mind, we shall impose the following conjecture.

\begin{conj}
In the notation of Conjecture \ref{R1-conj}, let $\Xi_E(\underline{\alpha},\underline{\beta}) $ and 
$\mathfrak{K}_E $ be as in \eqref{def-XiE-KE}. Then as $d$ ranges over $\mathcal{E}$, the joint distribution of $\log L(\oh,E_d)$ and $\log(|\Sha(E_d)|/ \sqrt{|d|})$
 is approximately bivariate with mean ${\bf 0}_{2}$ and covariance matrix $\mathfrak{K}_E $. More precisely, as $X\rightarrow \infty$,
\begin{align*}
\begin{split}
\#\bigg\{ d\in\mathcal{E},
20 <|d| \le X :\, & \frac{ \log L(\oh,E_{d}) + \frac{1}{2}\log \log |d| }{\sqrt{ \log\log |d|}}\in (\alpha_1,\beta_1),\\
& \frac{\log (\Sha(E_d)/\sqrt{|d|} ) - \mu(E) \log\log |d|}{\sqrt{\sigma(E)^2 \log\log |d|}}
\in (\alpha_2,\beta_2)\bigg\} 
\end{split}
\end{align*}
is asymptotic to
$
(\Xi_E(\underline{\alpha},\underline{\beta}) +o(1)) \# \{ d \in\mathcal{E} : 20< |d| \le X \}.  
$
\end{conj}

\section*{Acknowledgments}
The author thanks Amir Akbary, Po-Han Hsu, Nathan Ng, and Quanli Shen for their helpful comments and  suggestions.

\end{document}